\documentclass[11pt]{amsart}
\usepackage[T1]{fontenc}
\usepackage{amssymb, stmaryrd}
\usepackage{url}
\usepackage{bbm}

\usepackage{tikz}
\usetikzlibrary{tqft}
\usetikzlibrary{arrows}
\usetikzlibrary{cd}
\tikzset{>=latex}

\usepackage{hyperref}
\usepackage{amsrefs}

\newcommand{\stringdiagram}[1]{\[\begin{tikzpicture}[scale=.5, thick]
#1\end{tikzpicture}\]}
\newcommand{\stringdiagramlabel}[2]{\begin{equation}\label{#2}\begin{tikzpicture}[scale=.5, thick, baseline=(current  bounding  box.center)]
#1\end{tikzpicture}\end{equation}}
\newcommand{\identity}[2]{\draw (#1,#2+1) -- (#1,#2-1);}
\newcommand{\unit}[2]{\draw (#1,#2+.5) circle [radius=.1]; \draw (#1,#2+.4) -- (#1,#2-1);}

\newcommand{\multiplication}[2] {\draw (#1,#2) -- (#1-1,#2+1); \draw (#1,#2) -- (#1+1,#2+1); \draw (#1,#2) -- (#1,#2-1);}
\newcommand{\tripleprod}[2] {\draw (#1,#2) -- (#1-1,#2+1); \draw (#1,#2) -- (#1+1,#2+1); \draw (#1,#2+1) -- (#1,#2-1);}
\newcommand{\comultiplication}[2] {\draw (#1,#2) -- (#1-1,#2-1); \draw (#1,#2) -- (#1+1,#2-1); \draw (#1,#2) -- (#1,#2+1);}
\newcommand{\copairing}[2]{\draw (#1+1,#2-1) arc(0:180:1);}
\newcommand{\pairing}[2]{\draw (#1-1,#2+1) arc(180:360:1);}
\newcommand{\equals}[2]{\node at (#1,#2) {$=$};}

\makeatletter
\newcommand*{\da@rightarrow}{\mathchar"0\hexnumber@\symAMSa 4B }
\newcommand*{\xdashrightarrow}[2][]{%
  \mathrel{%
    \mathpalette{\da@xarrow{#1}{#2}{}\da@rightarrow{\,}{}}{}}}
\newcommand*{\da@xarrow}[7]{
\sbox0{$\ifx#7\scriptstyle\scriptscriptstyle\else\scriptstyle\fi#5#1#6\m@th$}
  \sbox2{$\ifx#7\scriptstyle\scriptscriptstyle\else\scriptstyle\fi#5#2#6\m@th$}
  \sbox4{$#7\dabar@\m@th$}
  \dimen@=\wd0 
  \ifdim\wd2 >\dimen@
    \dimen@=\wd2 
  \fi
  \count@=2 %
  \def\da@bars{\dabar@\dabar@}%
  \@whiledim\count@\wd4<\dimen@\do{
    \advance\count@\@ne \expandafter\def\expandafter\da@bars\expandafter{\da@bars\dabar@ }}
  \mathrel{#3}\mathrel{\mathop{\da@bars}\limits \ifx\\#1\\\else _{\copy0} \fi \ifx\\#2\\
   \else  ^{\copy2} \fi}   \mathrel{#4}}
\makeatother

\newcommand{\cat}{\mathcal{C}}
\DeclareMathOperator{\Ob}{Ob}
\DeclareMathOperator{\Mor}{Mor}
\newcommand{\catname}[1]{\mathbf{#1}}
\newcommand{\rel}{\catname{Rel}}
\newcommand{\srel}{\catname{SRel}}
\newcommand{\Span}{\catname{Span}}
\newcommand{\FSpan}{\catname{FSpan}}
\newcommand{\Symp}{\catname{Symp}}
\newcommand{\vect}{\catname{Vect}}

\newcommand{\id}{\mathrm{id}}

\newcommand{\reltolabel}[1]{\xdashrightarrow{#1}\!\!}
\newcommand{\tolabel}[1]{\xrightarrow{#1}}
\newcommand{\fromlabel}[1]{\xleftarrow{#1}}
\newcommand{\idx}{\mathbf{1}}

\newcommand{\arrows}{\,\substack{\to \\[-1em] \to}\,}
\newcommand{\threearrows}{\,\substack{\to \\[-1em] \to \\[-1em] \to}\,}
\newcommand{\isoto}{\stackrel{\sim}{\to}}
\DeclareMathOperator{\WW}{WW}
\newcommand{\fprod}{\star} 
\newcommand{\spanto}{\mathrel{\text{\rotatebox[origin=c]{33}{$\gets$}}\mspace{-6.9mu}\text{\rotatebox[origin=c]{-33}{$\to$}}}} 

\newcommand{\suchthat}{\mid}

\newcommand{\kk}{\Bbbk}
\newcommand{\bitimes}[2]{\,\vphantom{\times}_{#1}\!\! \times_{#2}\!}
\DeclareMathOperator{\im}{im}

\newcommand{\boundary}{\partial}
\newcommand{\mat}[1]{\begin{bmatrix}#1\end{bmatrix}}

\newtheorem{thm}{Theorem}[section]
\newtheorem{prop}[thm]{Proposition}
\newtheorem{lemma}[thm]{Lemma}
\newtheorem{cor}[thm]{Corollary}

\theoremstyle{definition}
\newtheorem{definition}[thm]{Definition}

\newtheorem{remark}[thm]{Remark}

\newtheorem{example}[thm]{Example}

\numberwithin{equation}{section}

\begin{document}

\title{Frobenius objects in the category of spans}
\author{Ivan Contreras}
\address{Department of Mathematics\\
Amherst College\\
31 Quadrangle Drive\\
Amherst, MA 01002}
\email{icontreraspalacios@amherst.edu}
\author{Molly Keller}
\author{Rajan Amit Mehta}
\address{Department of Mathematics \& Statistics\\
Smith College\\
44 College Lane\\
Northampton, MA 01063}
\email{mkeller@smith.edu}
\email{rmehta@smith.edu}


\begin{abstract}
We consider Frobenius objects in the category $\Span$, where the objects are sets and the morphisms are isomorphism classes of spans of sets. We show that such structures are in correspondence with data that can be characterized in terms of simplicial sets. An interesting class of examples comes from groupoids.

Our primary motivation is that $\Span$ can be viewed as a set-theoretic model for the symplectic category, and thus Frobenius objects in $\Span$ provide set-theoretic models for classical topological field theories. The paper includes an explanation of this relationship.

Given a finite commutative Frobenius object in $\Span$, one can obtain invariants of closed surfaces with values in the natural numbers. We explicitly compute these invariants in several examples, including examples arising from abelian groups.
\end{abstract}

\maketitle

\section{Introduction}

Frobenius algebras play an important role in the study of topological quantum field theory (TQFT). In particular, commutative Frobenius algebras are known to classify $2$-dimensional TQFTs \cites{abrams, dijkgraaf:thesis, kock-book}. Noncommutative Frobenius algebras appear in the classification of extended TQFTs \cite{schommer-pries:thesis}, open-closed TQFTs \cites{lauda-pfeiffer, BCT}, and lattice TQFTs \cites{lauda-pfeiffer:statesum, fhk}.

The appearance of Frobenius algebras in these contexts is closely related to the fact that Frobenius algebras can be defined completely in terms of algebraic data in the category of vector spaces. This fact allows the definition to be generalized by replacing the category of vector spaces with an arbitrary monoidal category $\cat$, giving the notion of a \emph{Frobenius object in $\cat$}. (This point of view appears, for example, in \cites{lauda-pfeiffer, lauda-pfeiffer:statesum, fuchs-stigner,kock-book, heunen-vicary:book, walton-yadav}.) If $\cat$ is symmetric monoidal, then one can define commutative Frobenius objects, which correspond to $2$-dimensional oriented topological field theories (TFTs) with values in $\cat$.

In this paper, we study Frobenius objects in the monoidal category $\Span$, where the objects are sets, the morphisms are isomorphism classes of spans of sets, and the monoidal structure is given by the Cartesian product. There are several reasons that $\Span$ is an interesting category to study in this context.
\begin{itemize}
    \item Our primary motivation is that $\Span$ is a good set-theoretic ``toy model'' for the Wehrheim-Woodward symplectic category \cite{ww}; this is explained in Section \ref{sec:spanww}. It is well-known that symplectic manifolds with compatible algebraic structures (such as symplectic groupoids \cite{cdw}) are closely connected to constructions of topological sigma models (such as the Poisson sigma model \cites{aksz, cattaneo-felder:poissonsigma,schaller-strobl}). It will be interesting to see a similar relationship established using functorial methods, and the present work provides a step in that direction.
    \item Since $\Span$ is symmetric monoidal, one can further consider commutative Frobenius objects, which classify $2$-dimensional TFTs with values in $\Span$. The resulting topological invariants of closed surfaces take values in the monoid of cardinalities. In particular, in the finite case, the invariants are natural numbers.
    
    There are simple examples where the invariants are sufficient to distinguish surfaces of different genus (see Section \ref{sec:examples}). This makes for a ``proof of concept'' that $\Span$ can be a useful target category for the more interesting case of $3$-dimensional TFTs.
    \item If we restrict to the subcategory of finite spans, then there is a symmetric monoidal functor to the category $\vect_\kk$ of vector spaces over a field $\kk$. Thus, a finite Frobenius object in $\Span$ gives a Frobenius algebra that is universal, in the sense that it is defined over any base field. Some well-known examples of Frobenius algebras, such as matrix algebras and group algebras, arise in this manner. 
\end{itemize}

Our main result is a one-to-one correspondence, up to isomorphism, between Frobenius objects in $\Span$ and simplicial sets that are equipped with an automorphism of the set of $1$-simplices, satisfying certain properties. In a sense, this correspondence gives us a way to relate TFTs to geometric structures. A simple example is that one can form a Frobenius object in $\Span$ from a group $G$, and the corresponding simplicial set is the nerve of $G$, whose realization is the classifying space $BG$. The automorphism of the set of $1$-simplices can be the inverse map, or more generally, it can be the inverse map ``twisted'' by an element of $G$.

More generally, one can form a Frobenius object in $\Span$ from a groupoid. This fact is related to \cite{hcc} (also see \cite{heunen-vicary:book}), where there is a correspondence between groupoids and special dagger Frobenius objects in $\rel$, the category of sets and relations, and \cite{Mehta-Zhang}, where Frobenius objects in $\rel$ were similarly described in terms of simplicial sets. However, the results here aren't strictly a generalization of the results of \cite{Mehta-Zhang}, since not every Frobenius object in $\rel$ can be lifted to a Frobenius object in $\Span$.

There is a close relationship between our work and that of Stern \cite{stern:2segal}, who proved an $\infty$-categorical equivalence between monoid objects in $\Span(\cat)$ and $2$-Segal simplicial objects in $\cat$, and between Calabi-Yau algebra objects in $\Span(\cat)$ and $2$-Segal cyclic objects in $\cat$. The objects that Stern studies are equipped with higher coherence data, making them more restrictive than ours, and thus our results don't directly follow from his. Nonetheless, Stern's work implies a rich source of examples of Frobenius objects in $\Span$, coming from $2$-Segal cyclic sets, and our work suggests directions in which Stern's results can be extended. We give more details about the relationship to $2$-Segal structures in Sections \ref{sec:2segal} and \ref{sec:symmetric}.

This work is also  related, at least in spirit, to recent work of Calaque, Haugseng, and Scheimbauer \cites{calaque:lagrangian, calaque:threelectures, Haugseng, CPTTV, CHS} 
on constructing (extended) TFTs with values in a version of the symplectic category from shifted symplectic structures via the AKSZ formalism. TFTs with values in $\Span$ reveal a shadow of these constructions in a context where it is easier to describe explicit examples.

The structure of the paper is as follows. In Section \ref{sec:spanandww}, we give a brief overview of the category $\Span$, and we explain how $\Span$ is related to the Wehrheim-Woodward construction. In Section \ref{sec:monoids}, we consider monoids in $\Span$, and we prove that they are in correspondence with simplicial sets satisfying certain properties. In Section \ref{sec:frob}, we consider Frobenius objects in $\Span$, and we prove that a Frobenius structure on a monoid can be encoded by an automorphism of the set of $1$-simplices of the corresponding simplicial set, satisfying certain properties. In Section \ref{sec:examples}, we describe some examples of commutative Frobenius objects in $\Span$, including examples associated to abelian groups and two infinite families of $2$-element examples, and we find explicit formulas for the associated topological invariants. Finally, in Section \ref{sec:vect}, we briefly describe how finite Frobenius objects in $\Span$ give rise to Frobenius algebras.

\subsection*{Acknowledgements}
The authors would like to thank Chris Heunen, Adele Long, Sophia Marx, and Sasha Peterka for helpful conversations on topics related to the paper. We especially thank Walker Stern for suggesting the relationship to $2$-Segal conditions, and for helping us figure out the specifics thereof. I.C. and R.A.M. would also like to thank the LEGO Group for entertaining our children long enough for us to finish this work during the pandemic.

\section{Spans and the Wehrheim-Woodward construction}\label{sec:spanandww}

In this section, we review the definition of the category of spans, and we explain why it can be viewed as a set-theoretic model for the symplectic category. Some references for spans are \cites{benabou, cks, dawson-pare-pronk:universal}.

\subsection{The category of spans}\label{sec:span}

Given two sets $X,Y$, a \emph{span} from $X$ to $Y$ is a triplet $(A,f_1,f_2)$, where $A$ is a set, $f_1$ is a map from $A$ to $X$, and $f_2$ is a map from $A$ to $Y$. It can be helpful to visualize a span as a diagram as in Figure \ref{fig:span}.
    \begin{figure}[th]
    \begin{tikzcd}
    & A \arrow{dr}{f_2} \arrow[swap]{dl}{f_1} & \\
    X & & Y
    \end{tikzcd}
        \caption{A span.}
        \label{fig:span}
    \end{figure}

Two spans $(A,f_1,f_2)$, $(A',f_1',f_2')$ from $X$ to $Y$ are considered isomorphic if there exists a bijection $\phi: A \to A'$ such that the diagram
    \begin{equation}
    \label{diag:isospan}
    \begin{tikzcd}
    & A \arrow[swap]{drr}{f_2} \arrow[swap]{dl}{f_1} \arrow{r}{\phi} & A' \arrow{dr}{f_2'} \arrow{dll}{f_1'} & \\
    X & & & Y
    \end{tikzcd}
    \end{equation}
commutes.

The category of spans, denoted $\Span$, is defined as follows. 
\begin{itemize}
    \item The objects of $\Span$ are sets.
    \item If $X$ and $Y$ are sets, then a morphism from $X$ to $Y$ is an isomorphism class of spans from $X$ to $Y$.
    \item Composition of morphisms is given by pullback; see Figure \ref{fig:composition}.
    \begin{figure}[th]
        \begin{tikzcd}
    && A \bitimes{f_2}{g_1} B \arrow[swap]{dl}{p_1} \arrow{dr}{p_2} && \\
    & A \arrow[swap]{dl}{f_1} \arrow{dr}{f_2} && B \arrow[swap]{dl}{g_1} \arrow{dr}{g_2} & \\
    X && Y && Z
    \end{tikzcd}
        \caption{Composition of spans.}
        \label{fig:composition}
    \end{figure}
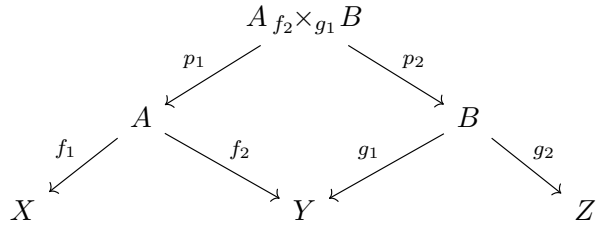
    \item The Cartesian product endows $\Span$ with the structure of a monoidal category.
\end{itemize}

We use the notation $f:X \spanto Y$ to denote a morphism $f$ from $X$ to $Y$ in $\Span$. Because the span representing a morphism in $\Span$ is only well-defined up to isomorphism, it can be difficult to extract equations involving maps of sets from equations involving morphisms in $\Span$. However, there are some identities in $\Span$ that can be upgraded to identities at the level of sets. 

For any set $X$, the identity morphism $\idx: X \spanto X$ has a canonical representative
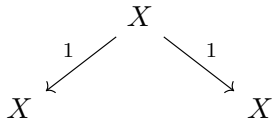
\[ 
\begin{tikzcd}
& X \arrow[swap]{dl}{1} \arrow{dr}{1} & \\
X && X
\end{tikzcd}
\]
through which all other representatives uniquely factor. Furthermore, the compositions of a span $(A, f_1, f_2)$ with identity morphisms have natural representatives
\[
\begin{tikzcd}
  && A \arrow[swap]{dl}{1} \arrow{dr}{f_2}&& \\
  & A \arrow[swap]{dl}{f_1} \arrow{dr}{f_2} && Y \arrow[swap]{dl}{1} \arrow{dr}{1} & \\
  X && Y && Y
\end{tikzcd}
\]
and
\[
\begin{tikzcd}
  && A \arrow[swap]{dl}{f_1} \arrow{dr}{1}&& \\
  & X \arrow[swap]{dl}{1} \arrow{dr}{1} && A \arrow[swap]{dl}{f_1} \arrow{dr}{f_2} & \\
  X && X && Y
\end{tikzcd}
\]
where the identity map on $A$ gives an isomorphism with $(A, f_1, f_2)$.

One of the properties of a monoidal category is the following ``slide move''. If $\alpha: X \spanto Y$ and $\beta: W \spanto Z$ are morphisms in $\Span$, then $(\idx \times \beta) \circ (\alpha \times \idx) = \alpha \times \beta = (\alpha \times \idx) \circ (\idx \times \beta)$. If $\alpha$ and $\beta$ are represented by $(A, f_1, f_2)$ and $(B, g_1, g_2)$, respectively, then $(\idx \times \beta) \circ (\alpha \times \idx)$ is represented by
\[
\begin{tikzcd}
    && A \times B \arrow[swap]{dl}{1 \times g_1} \arrow{dr}{f_2 \times 1} && \\
    & A \times W \arrow[swap]{dl}{f_1 \times 1} \arrow{dr}{f_2 \times 1} && Y \times B \arrow[swap]{dl}{1 \times g_1} \arrow{dr}{1 \times g_2} & \\
    X \times W && Y \times W && Y \times Z
\end{tikzcd}
\]
and $(\alpha \times \idx) \circ (\idx \times \beta)$ is represented by
\[
\begin{tikzcd}
    && A \times B \arrow[swap]{dl}{f_1 \times 1} \arrow{dr}{1 \times g_2} && \\
    & X \times B \arrow[swap]{dl}{1 \times g_1} \arrow{dr}{1 \times g_2} && A \times Z \arrow[swap]{dl}{f_1 \times 1} \arrow{dr}{f_2 \times 1} & \\
    X \times W && X \times Z && Y \times Z
\end{tikzcd}
\]
where the identity map on $A \times B$ gives an isomorphism between the above two compositions. The conclusion is that, if we can recognize two spans as being related via composition with the identity and/or the slide move, as given above, then we know that the spans are not only isomorphic, but equal.

Finally, the following basic fact about $\Span$ will be used later in the paper. We note that this proof is similar to that of \cite{Haugseng}*{Lemma 8.2}, but specialized to the present situation.
\begin{prop}\label{prop:spaniso}
A span $(A, f_1, f_2)$ represents an isomorphism in $\Span$ if and only if $f_1$ and $f_2$ are bijections.
\end{prop}

\begin{proof}
A span $(A, f_1, f_2)$ from $X$ to $Y$ represents an isomorphism if and only if there exists a span $(B, g_1, g_2)$ from $Y$ to $X$ such that the compositions
\begin{equation}\label{diag:spaniso1}
\begin{tikzcd}
&& A \bitimes{f_2}{g_1} B \arrow[swap]{dl}{p_1} \arrow{dr}{p_2} && \\
& A \arrow[swap]{dl}{f_1} \arrow{dr}{f_2} && B \arrow[swap]{dl}{g_1} \arrow{dr}{g_2} \\
X && Y && X
\end{tikzcd}
\end{equation}
and 
\begin{equation}\label{diag:spaniso2}
\begin{tikzcd}
&& B \bitimes{g_2}{f_1} A \arrow[swap]{dl}{q_1} \arrow{dr}{q_2} && \\
& B \arrow[swap]{dl}{g_1} \arrow{dr}{g_2} && A \arrow[swap]{dl}{f_1} \arrow{dr}{f_2} \\
Y && X && Y
\end{tikzcd}
\end{equation}
are isomorphic to the identity spans on $X$ and $Y$, respectively. Here, we use $p_i$ and $q_i$ to denote the projection maps onto the $i$th component.

The composition \eqref{diag:spaniso1} is isomorphic to the identity span on $X$ if and only if there is a bijection $\phi: X \to A \bitimes{f_2}{g_1} B$ such that 
\begin{equation}\label{eqn:spaniso1}
    f_1 p_1 \phi = g_2 p_2 \phi = 1.
\end{equation}
Similarly, \eqref{diag:spaniso2} is isomorphic to the identity span on $Y$ if and only if there is a bijection $\phi': Y \to B \bitimes{g_2}{f_1} A$ such that 
\begin{equation}\label{eqn:spaniso2}
    g_1 q_1 \phi' = f_2 q_2 \phi' = 1.
\end{equation}
If $f_1$ and $f_2$ are bijections, then we can take $B=A$, $g_1 = f_2$, $g_2 = f_1$, and
\begin{align*}
    \phi &= (f_1^{-1},f_1^{-1}): X \to A \bitimes{f_2}{f_2} A,\\
    \phi' &= (f_2^{-1},f_2^{-1}): Y \to A \bitimes{f_1}{f_1} A,
\end{align*}
and \eqref{eqn:spaniso1} and \eqref{eqn:spaniso2} will be satisfied.

On the other hand, given bijections $\phi$ and $\phi'$ satisfying \eqref{eqn:spaniso1} and \eqref{eqn:spaniso2}, it follows that $f_1$, $f_2$, $g_1$, and $g_2$ are surjective, and that $p_1$, $p_2$, $q_1$, and $q_2$ are injective. Together, the injectivity of $p_2$ and the surjectivity of $g_1$ imply that $f_2$ is injective. Similarly, the injectivity of $q_1$ and the surjectivity of $g_2$ imply that $f_1$ is injective. Thus, $f_1$ and $f_2$ are bijections.
\end{proof}

\subsection{The Wehrheim-Woodward construction}\label{sec:spanww}

Given two sets $X,Y$, a \emph{relation} from $X$ to $Y$ is a subset of $X \times Y$. There is a category, denoted $\rel$, for which the objects are sets and the morphisms are relations. There is a natural functor from $\Span$ to $\rel$, where a span $(A,f,g)$ is sent to the relation $\{(f(a),g(a)) \suchthat a \in A\}$.

The functor from $\Span$ to $\rel$ is only a part of a more sophisticated relationship that was discovered by Li-Bland and Weinstein \cite{li-bland-weinstein} as part of an effort to better understand the symplectic category. We provide a very brief review here, referring to \cite{li-bland-weinstein} for details.

A \emph{selective category} is a category with a distinguished class of morphisms, called \emph{suave}, and a distinguished class of composable pairs of suave morphisms, called \emph{congenial}, satisfying certain axioms. The examples that are relevant for the present work are as follows:
\begin{itemize}
    \item $\rel$, where all of the morphisms are suave, and the congenial pairs are those that are monic. To be precise, if $R \subseteq X \times Y$ is a relation from $X$ to $Y$ and $S \subseteq Y \times Z$ is a relation from $Y$ to $Z$, then the pair $(S,R)$ is \emph{monic} if, for all $(x,z) \in X \times Z$, there exists at most one $y \in Y$ such that $(x,y) \in R$ and $(y,z) \in S$.
    \item $\srel$ is the category where the objects are symplectic manifolds, and where the morphisms are all set-theoretic relations. The suave morphisms are the Lagrangian relations, and the congenial pairs are those that are strongly transversal \cite{weinstein:ww}.
\end{itemize}

Li-Bland and Weinstein defined a construction that takes a selective category $\cat$ and constructs a new category $\WW(\cat)$. The objects of $\WW(\cat)$ are the same as those of $\cat$. If $X$ and $Y$ are objects, a morphism from $X$ to $Y$ is a finite sequence of suave morphisms
\[ X = X_0 \tolabel{f_1} X_1 \tolabel{f_2} \cdots \tolabel{f_n} X_n = Y,\]
modulo the following equivalence relation. If we write such a sequence as a formal product $f_n\fprod \cdots \fprod f_1$, then the equivalence relation is generated by relations of the form $f_{i+1} \fprod f_i = f_{i+1} f_i$ for congenial pairs $(f_{i+1}, f_i)$.

In the case where $\cat = \srel$, this construction reproduces the Wehrheim-Woodward category \cite{ww}, which is a rigorous construction of the symplectic category $\Symp$. The following results appear in \cite{li-bland-weinstein}:
\begin{itemize}
    \item There is a canonical isomorphism between $\WW(\rel)$ and $\Span$. We review this isomorphism below in Section \ref{sec:iso}.
    \item The forgetful functor $\srel \to \rel$ is compatible with the selective structures, i.e.\ it sends suave morphisms to suave morphisms and congenial pairs to congenial pairs. Therefore, there is an induced functor $\Symp = \WW(\srel) \to \WW(\rel) = \Span$.
\end{itemize}
If $\cat$ is a selective category, then there is a \emph{composition functor} $\WW(\cat) \to \cat$, taking $f_n \fprod \cdots \fprod f_1$ to $f_n \cdots f_1$. Via the isomorphism $\WW(\rel) \isoto \Span$, this recovers the functor $\Span \to \rel$ described at the beginning of this section.

The above discussion can be summarized by the following sequence of functors:
\[ \Symp \to \Span \to \rel.\]
This relationship explains our assertion that $\Span$ is a good set-theoretic model for $\Symp$, incorporating aspects of the Wehrheim-Woodward category that don't appear in $\rel$.

\subsection{\texorpdfstring{$\WW(\rel) = \Span$}{WW(Rel)=Span}}\label{sec:iso}
In this section, we will describe an isomorphism between $\WW(\rel)$ and $\Span$, using an approach that is slightly more direct than the one in \cite{li-bland-weinstein}. 

Recall that, if $X$ and $Y$ are sets, then a relation from $X$ to $Y$ is a subset $R \subseteq X \times Y$. When viewing such a subset as a morphism in $\rel$, we will denote it as $X \reltolabel{R} Y$. 

The objects of $\WW(\rel)$ are sets, and a morphism from $X$ to $Y$ is represented by a finite sequence of relations
\begin{equation}\label{eqn:wwrel}
    X = X_0 \reltolabel{R_1} X_1 \reltolabel{R_2} \cdots \reltolabel{R_n} X_n = Y,
\end{equation}
viewed as a formal product $R_n \fprod \cdots \fprod R_1$, modulo the relation $R_{i+1} \fprod R_i = R_{i+1} R_i$ when the pair $(R_{i+1}, R_i)$ is monic.

If $f:X \to Y$ is a map of sets, then the graph of $f$ is a relation from $X$ to $Y$. If $g:Y \to X$ is a map of sets, then the graph of $g$ can also be viewed as a relation from $X$ to $Y$. In these cases, we will denote the corresponding relations by $X \tolabel{f} Y$ and $X \fromlabel{g} Y$. If $(A, f_1, f_2)$ is a span from $X$ to $Y$, then $X \fromlabel{f_1} A \tolabel{f_2} Y$ represents a morphism in $\WW(\rel)$ from $X$ to $Y$. 

\begin{prop}\label{prop:functor}
The map $(A, f_1, f_2) \mapsto (X \fromlabel{f_1} A \tolabel{f_2} Y)$ defines a functor $F: \Span \to \WW(\rel)$.
\end{prop}
\begin{proof}
First, we observe that $F$ is well-defined at the level of morphisms, i.e.\ that isomorphic spans are mapped to equivalent sequences of relations. Specifically, given an isomorphism of spans as in \eqref{diag:isospan}, we see that ${X \fromlabel{f_1'} A' \tolabel{f_2'} Y}$ is equivalent to $X \fromlabel{f_1} A \tolabel{f_2} Y$ since $f_1' \phi = f_1$ and $f_2' \phi = f_2$ are monic compositions.

Next, we consider a composition of spans as in Figure \ref{fig:composition}. We observe that $A \tolabel{f_2} Y \fromlabel{g_1} B$ and $A \fromlabel{p_1} A \bitimes{f_2}{g_1} B \tolabel{p_2} B$ are both monic compositions that give rise to the same relation $A \reltolabel{R} B$, where $R = \{(a,b) \suchthat f_2(a) = g_1(b)\}$. Thus, $X \fromlabel{f_1 p_1} A \bitimes{f_2}{g_1} B \tolabel{g_2 p_2} Z$ is equivalent to $X \fromlabel{f_1} A \tolabel{f_2} Y \fromlabel{g_1} B \tolabel{g_2} Z$, which proves that $F$ preserves compositions.

Finally, we observe that the identity span on $X$ is sent to $X \fromlabel{1} X \tolabel{1} X$, which is a monic composition that gives rise to the identity relation $X \reltolabel{1} X$.
\end{proof}

\begin{prop}
The functor $F$ in Proposition \ref{prop:functor} is an isomorphism of categories.
\end{prop}
\begin{proof}
Since $F$ is the identity on objects, it suffices to prove that $F$ is full and faithful.

Any relation $X \reltolabel{R} Y$ can be factored as a monic composition $X \gets R \to Y$ and is thus in the image of $F$. Since the morphisms in $\WW(\rel)$ are generated by relations, it follows that $F$ is full.

Given a sequence of relations as in \eqref{eqn:wwrel}, we can obtain a span as follows. Let $A = A(R_1, \dots, R_n)$ be the set consisting of all \emph{trajectories} from $X$ to $Y$, i.e.\ sequences $(x_0,\dots, x_n)$, $x_i \in X_0$, such that $(x_{i-1}, x_i) \in R_i$ for $i = 1, \dots, n$. Projection onto the first and last components give a span $X \gets A \to Y$. 

If $(R_{i+1}, R_i)$ is a monic pair, then the contraction $(x_0, \dots, x_n) \mapsto (x_0, \dots \hat{x}_i, \dots, x_n)$ gives an isomorphism of spans from $A(R_1, \dots, R_n)$ to $A(R_1, \dots, R_{i+1}R_i, \dots R_n)$. Thus the map $R_n \fprod \cdots \fprod R_1 \mapsto A(R_1, \dots R_n)$ gives a well-defined map from morphisms in $\WW(\rel)$ to morphisms in $\Span$. This map is a left inverse to $F$, demonstrating that $F$ is faithful.
\end{proof}

As an aside, we note that both $\WW(\rel)$ and $\Span$ can be viewed as truncations of $2$-categories, so it seems reasonable to expect that this isomorphism can be upgraded to an equivalence of $2$-categories.

\section{Monoids in \texorpdfstring{$\Span$}{Span} and simplicial sets}\label{sec:monoids}

If $\cat$ is a monoidal category, then one can define the notion of a \emph{monoid} in $\cat$. In this section, we study monoids in $\Span$. We will see that there is a correspondence, up to isomorphism, between monoids in $\Span$ and simplicial sets satisfying certain properties.

\subsection{Monoids in \texorpdfstring{$\Span$}{Span}}

A monoid in $\Span$ consists of a set $X$, equipped with a morphism $\eta: \bullet \spanto X$ (\emph{unit}) and a morphism $\mu: X \times X \spanto X$ (\emph{multiplication}), satisfying
\begin{enumerate}
    \item (unit axiom) $\mu \circ (\idx \times \eta) = \mu \circ (\eta \times \idx) = \idx$,
    \item (associativity) $\mu \circ (\idx \times \mu) = \mu \circ (\mu \times \idx)$.
\end{enumerate}
Here, $\bullet$ denotes a set containing one element, and $\idx$ denotes the identity morphism on $X$.

It is frequently convenient to use string diagrams to depict morphisms that are constructed out of the unit, multiplication, and identity via composition and monoidal product. Specifically, $\eta$, $\mu$, and $\idx$ are respectively depicted as follows, where the diagrams should be read from top to bottom:
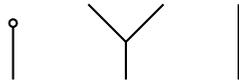
\stringdiagram{
\unit{-3}{0}
\multiplication{0}{0}
\identity{3}{0}
}
The unit and associativity axioms can then be depicted as follows:
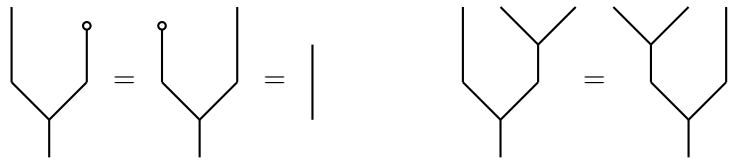
\stringdiagram{
\begin{scope} 
\unit{-2}{1}
\identity{-4}{1}
\multiplication{-3}{-1}
\equals{-1}{0}
\unit{0}{1}
\identity{2}{1}
\multiplication{1}{-1}
\equals{3}{0}
\identity{4}{0}
\end{scope}

\begin{scope}[shift={(8,0)}] 
\identity{0}{1}
\multiplication{2}{1}
\multiplication{1}{-1}
\equals{3.5}{0}
\multiplication{5}{1}
\identity{7}{1}
\multiplication{6}{-1}
\end{scope}
}

\subsection{Simplicial sets}\label{sec:simplicial}

To establish notation and terminology, we briefly review some definitions relating to simplicial sets.

\begin{definition} \label{dfn:simplicial}A \emph{simplicial set} $\mathcal{X}$ is a sequence $X_0, X_1, \dots$  of sets equipped with maps $d_i^q: X_q \to X_{q-1}$ (called \emph{face maps}), $0 \leq i \leq q$, and $s_i^q : X_q \to X_{q+1}$ (called \emph{degeneracy maps}), $0 \leq i \leq q$,  such that
\begin{align}
d_i^{q-1}d_j^q &= d_{j-1}^{q-1}d_i^q, \qquad i < j,\label{eqn:twoface}\\
s_i^{q+1}s_j^q &= s_{j+1}^{q+1}s_i^q, \qquad i \leq j, \label{eqn:twodegen}\\
d_i^{q+1}s_j^q &= \begin{cases}
s_{j-1}^{q-1}d_i^q, & i< j,\\
\id, & i = j \mbox{ or }j+1, \\
s_j^{q-1}d_{i-1}^q, & i > j+1. \end{cases}\label{eqn:facedegen}
\end{align}
\end{definition}

We will also need to consider $n$-truncated versions of simplicial sets, which only include data going up to $X_n$:
\begin{definition}\label{dfn:nsimplicial}
An \emph{$n$-truncated simplicial set} $\mathcal{X}$ is a sequence $X_0, X_1, \dots X_n$ of sets equipped with face maps $d_i^q: X_q \to X_{q-1}$, $0 \leq i \leq q \leq n$, and degeneracy maps $s_i^q : X_q \to X_{q+1}$, $0 \leq i \leq q < n$, satisfying \eqref{eqn:twoface}--\eqref{eqn:facedegen} whenever both sides of an equation are defined.
\end{definition}

Suppose that $\mathcal{X}$ is a (possibly $n$-truncated) simplicial set. For $1 \leq q \leq n+1$, let $\Delta_q \mathcal{X}$ denote the set of $(q+1)$-tuples $(\zeta_0, \dots \zeta_q)$, $\zeta_i \in X_{q-1}$, such that
\begin{equation}\label{eqn:horncompat}
     d_i^{q-1} \zeta_j = d_{j-1}^{q-1} \zeta_i
\end{equation}
for $i < j$. There is a natural \emph{boundary map} $\delta^q: X_q \to \Delta_q \mathcal{X}$, given by
\[ \delta^q(w) = (d_0^q w, \dots, d_q^q w).\]
\begin{definition}
A simplicial set $\mathcal{X}$ is called \emph{$n$-coskeletal} if $\delta^q$ is a bijection for $q>n$.
\end{definition}
It is well-known (see, for example, \cite{artin-mazur}) that any $n$-truncated simplicial set has a unique extension to an $n$-coskeletal simplicial set. The extension can be recursively constructed by taking $X_{n+1} = \Delta_{n+1} \mathcal{X}$.

\subsection{From simplicial sets to monoids}

Let $X_\bullet$ be a $2$-truncated simplicial set. Without any further assumptions, we can construct the spans
\begin{equation}\label{diag:simp2monoid}
\begin{tikzcd}
& X_0 \arrow{dl} \arrow{dr}{s_0^0} & \\
\bullet && X_1
\end{tikzcd} 
\;\;\;\;\;
\begin{tikzcd}
& X_2 \arrow[swap]{dl}{(d_2^2,d_0^2)} \arrow{dr}{d_1^2} & \\
X_1 \times X_1 && X_1
\end{tikzcd}
\end{equation}
which respectively represent morphisms $\eta: \bullet \spanto X_1$ and $\mu: X_1 \times X_1 \spanto X_1$ in $\Span$. We can then ask whether $(X_1, \eta, \mu)$ satisfies the axioms of a monoid in $\Span$. The following lemmas establish necessary and sufficient conditions.

\begin{lemma}\label{lemma:simpunit}
The unit axiom holds if and only if the following conditions hold for all $\zeta \in X_2$:
\begin{enumerate}
    \item If $d_2^2 \zeta \in \im(s_0^0)$, then $\zeta \in \im(s_0^1)$.
    \item If $d_0^2 \zeta \in \im(s_0^0)$, then $\zeta \in \im(s_1^1)$.
\end{enumerate}
\end{lemma}
\begin{proof}
Consider the composition $\mu \circ (\eta \times \idx)$:
\[
\begin{tikzcd}
&&[-20pt] (X_0 \times X_1) * X_2 \arrow{dl} \arrow{dr} &[-15pt]& \\
& X_0 \times X_1 \arrow[swap]{dl}{p_2} \arrow{dr}{s_0^0 \times 1} & & X_2 \arrow[swap]{dl}{(d_2^2,d_0^2)} \arrow{dr}{d_1^2} &\\
X_1 && X_1 \times X_1 && X_1
\end{tikzcd}
\]
Here, $p_2$ is projection onto the second component. 

The pullback in the above diagram is
\[ (X_0 \times X_1) * X_2 = \{(u,x,\zeta) \in X_0 \times X_1 \times X_2 \suchthat s_0^0 u = d_2^2 \zeta, \; x = d_0^2 \zeta \},\]
which, since $s_0^0$ is injective, can be identified with 
\[ A := \{\zeta \in X_2 \suchthat d_2^2 \zeta \in \im(s_0^0)\} \subseteq X_2.\]
The equation $\mu \circ (\eta \times \idx) = \idx$ holds if and only if there is a bijection $\phi: A \to X_1$ such that $\phi(\zeta) = d_0^2 \zeta = d_1^2 \zeta$. This condition completely determines $\phi$, so the question is whether it is well-defined, i.e.\ if $d_0^2 \zeta = d_1^2 \zeta$, and if it is bijective.

For all $x \in X_1$, the degenerate $2$-simplex $s_0^1 x$ satisfies
\begin{align*}
d_2^2 s_0^1 x &= s_0^0 d_1^1 x, & d_0^2 s_0^1 x = d_1^2 s_0^1 x &= x.
\end{align*} 
Thus, $s_0^1 x$ is in $A$, and its image under $\phi$ is well-defined and equal to $x$. It follows that $\phi$ is well-defined and bijective if and only if $A = \{s_0^1 x \suchthat x \in X_1\}$, which is equivalent to condition (1) of the lemma.

The proof that the equation $\mu \circ (\idx \times \eta) = \idx$ is equivalent to condition (2) of the lemma is similar.
\end{proof}

To describe associativity in the language of simplicial sets, we will use the following notion. For $0 \leq i < j \leq 3$, the set $T_{ij} \mathcal{X}$ of \emph{$(ij)$-tacos} is defined as
\[ T_{ij}\mathcal{X} = \{(\zeta, \zeta') \in X_2 \times X_2 \suchthat d_{j-1}^2 \zeta = d_i^2 \zeta'\}.\]
Geometrically, an element of $T_{ij}\mathcal{X}$ is a pair of $2$-simplices that share an edge in a way that allows them to form the $i$th and $j$th faces of a $3$-simplex; see Figure \ref{fig:taco}. 

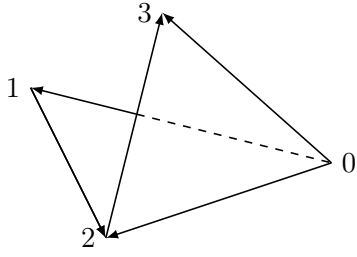
\begin{figure}[th]
\begin{tikzpicture}
\begin{scope}[semithick]
\coordinate [label=left:2] (b) at (0,0);
\coordinate [label=left:1] (c) at (-1,2);
\coordinate [label=left:3] (a) at (.75,3);
\coordinate [label=right:0] (d) at (3,1);
\coordinate (e) at (intersection of c--d and a--b);

\draw[<-] (a) -- (b);
\draw[<-] (a) -- (d);
\draw[<-] (b) -- (c);
\draw[<-] (b) -- (d);
\draw[<-] (b) -- (c);
\draw [<-] (c) -- (e);
\draw[dashed] (e) -- (d);
\end{scope}
\end{tikzpicture}
    \caption{A $(13)$-taco.}
    \label{fig:taco}
\end{figure}

The boundary of a taco consists of four $1$-simplices that satisfy certain compatibility conditions. In particular, the $(02)$-tacos and the $(13)$-tacos are complementary, so their boundaries have the same compatibility conditions. Specifically, let $S\mathcal{X}$ be the subset of $(X_1)^4$ consisting of all $(x_{01}, x_{12}, x_{23}, x_{03})$ such that
\begin{align*}
    d_0^1 x_{01} &= d_1^1 x_{12}, & d_0^1 x_{12} &= d_1^1 x_{23}, \\
    d_0^1 x_{23} &= d_0^1 x_{03}, & d_1^1 x_{03} &= d_1^1 x_{01}.
\end{align*}
The boundary maps $\boundary_{02}: T_{02}\mathcal{X} \to S\mathcal{X}$ and $\boundary_{13}: T_{13}\mathcal{X} \to S\mathcal{X}$ are defined by 
\begin{align*}
\boundary_{02}(\zeta, \zeta') &= (d_2^2 \zeta', d_2^2 \zeta, d_0^2 \zeta, d_1^2 \zeta'),\\
\boundary_{13}(\zeta, \zeta') &= (d_2^2 \zeta', d_0^2 \zeta', d_0^2 \zeta, d_1^2 \zeta).
\end{align*}

\begin{lemma}\label{lemma:simpassociativity}
The associativity axiom holds if and only if there exists a bijection $T_{02}\mathcal{X} \cong T_{13}\mathcal{X}$ that commutes with the boundary maps to $S\mathcal{X}$.
\end{lemma}

\begin{proof}
The composition $\mu \circ (\idx \times \mu)$ is as follows:
\begin{equation}\label{diag:t02}
\begin{tikzcd}
    &[-15pt] &[-15pt] (X_1 \times X_2) * X_2 \arrow{dl} \arrow{dr} &[-15pt]& \\
    & X_1 \times X_2 \arrow[swap]{dl}{1 \times (d_2^2, d_0^2)} \arrow{dr}{1 \times d_1^2} && X_2 \arrow[swap]{dl}{(d_2^2,d_0^2)} \arrow{dr}{d_1^2}& \\
    X_1 \times X_1 \times X_1 && X_1 \times X_1 && X_1
\end{tikzcd}
\end{equation}
The pullback in \eqref{diag:t02} is
\begin{align*} 
(X_1 \times X_2) * X_2 &= \{(x,\zeta,\zeta') \in X_1 \times X_2 \times X_2 \suchthat x = d_2^2 \zeta', \; d_1^2 \zeta = d_0^2 \zeta'\} \\
&\cong \{(\zeta,\zeta') \in X_2 \times X_2 \suchthat d_1^2 \zeta = d_0^2 \zeta'\}\\
&\cong T_{02}\mathcal{X}.
\end{align*}

On the other hand, the composition $\mu \circ (\mu \times \idx)$ is as follows:
\begin{equation}\label{diag:t13}
\begin{tikzcd}
    &[-15pt]&[-15pt] (X_2 \times X_1) * X_2 \arrow{dl} \arrow{dr} &[-15pt]& \\
    & X_2 \times X_1 \arrow[swap]{dl}{(d_2^2, d_0^2) \times 1} \arrow{dr}{d_1^2 \times 1} && X_2 \arrow[swap]{dl}{(d_2^2,d_0^2)} \arrow{dr}{d_1^2}& \\
    X_1 \times X_1 \times X_1 && X_1 \times X_1 && X_1
\end{tikzcd}
\end{equation}
The pullback in \eqref{diag:t13} is
\begin{align*} 
(X_2 \times X_1) * X_2 &= \{(\zeta'',x,\zeta''') \in X_2 \times X_1 \times X_2 \suchthat x = d_0^2 \zeta''', \; d_1^2 \zeta'' = d_2^2 \zeta'''\} \\
&\cong \{(\zeta'',\zeta''') \in X_2 \times X_2 \suchthat d_1^2 \zeta'' = d_2^2 \zeta'''\} \\
&\cong T_{13}\mathcal{X}.
\end{align*}
Note that the isomorphism with $T_{13}\mathcal{X}$ involves swapping the two components.

The associativity axiom holds if and only if the spans \eqref{diag:t02} and \eqref{diag:t13} are isomorphic, i.e.\ there exists a bijection $T_{02}\mathcal{X} \to T_{13}\mathcal{X}$, $(\zeta, \zeta') \mapsto (\zeta''', \zeta'')$, such that
\begin{align*}
d_2^2 \zeta' &= d_2^2 \zeta '', & d_2^2 \zeta =& d_0^2 \zeta'', & d_0^2 \zeta &= d_0^2 \zeta''', & d_1^2 \zeta' &= d_1^2 \zeta''',
\end{align*}
or, equivalently, $\boundary_{02}(\zeta,\zeta') = \boundary_{13}(\zeta''', \zeta'')$.
\end{proof}

Together, Lemmas \ref{lemma:simpunit} and \ref{lemma:simpassociativity} give us the following the result.

\begin{thm}\label{thm:simp2monoid}
Suppose $X_\bullet$ is a $2$-truncated simplicial set, and let $\eta: \bullet \spanto X_1$ and $\mu: X_1 \times X_1 \spanto X_1$ be given by the spans in \eqref{diag:simp2monoid}. Then $(X_1, \eta, \mu)$ is a monoid in $\Span$ if and only if the conditions in Lemmas \ref{lemma:simpunit} and \ref{lemma:simpassociativity} hold.
\end{thm}

\begin{example}\label{ex:cat}
Let $\cat$ be a small category. As is well-known, there is an associated simplicial set $N\cat_\bullet$, called the \emph{nerve} of $\cat$. The nerve of $\cat$ is $2$-coskeletal, with $N\cat_0 = \Ob(\cat)$, $N\cat_1 = \Mor(\cat)$, and
\[ N\cat_2 = \{(f,g) \in \Mor(\cat) \suchthat s(f) = t(g)\},\]
where $s,t : \Mor(\cat) \to \Ob(\cat)$ are the source and target maps. The face and degeneracy maps are as follows:
\begin{align*}
    d_0^1 &= s, & d_1^1 &= t, &s_0^0(x) &= 1_x,\\
    d_0^2(f,g) &= g, & d_1^2(f,g) &= fg, & d_2^2(f,g) &= f, \\
    s_0^1(f) &= (1_{t(f)}, f), & s_1^1(f) &= (f,1_{s(f)}).
\end{align*}
One can directly see that $N\cat_\bullet$ satisfies the condition of Lemma \ref{lemma:simpunit}. Furthermore, in this case, $T_{02}N\cat$ and $T_{13}N\cat$ can both be canonically identified with $N\cat_3$, giving the isomorphism for the condition in Lemma \ref{lemma:simpassociativity}.
\end{example}

\begin{example}\label{ex:twoelement}
Consider the case where $X_0 = \{e\}$ has one point, $X_1=\{a,b\} $ has two points, and $X_2$ is finite. Here we will describe all possible $2$-truncated simplicial sets of this form that satisfy the conditions of Theorem \ref{thm:simp2monoid}.

Without loss of generality, we can take $s_0^0(e) = a$. We can decompose $X_2$ into a disjoint union of $Y_{ijk}$ for $i,j,k \in X_1$, where $d_0(\zeta) = i$, $d_1(\zeta) = j$, and $d_2(\zeta) = k$ for $\zeta \in Y_{ijk}$. Then $X_2$ is determined up to isomorphism by the cardinalities $n_{ijk}$ of $Y_{ijk}$.

From the simplicial axioms, we have that $s_0^1(a) = s_1^1(a) \in Y_{aaa}$, $s_0^1(b) \in Y_{bba}$, and $s_1^1(b) \in Y_{abb}$. 

Condition (1) in Lemma \ref{lemma:simpunit} says that, if $d_2^2 \zeta = a$, then $\zeta = s_0^1(a)$ or $\zeta = s_0^1(b)$. Thus, $n_{aaa} = n_{bba} = 1$, and $n_{aba} = n_{baa} = 0$. Similarly, condition (2) in Lemma \ref{lemma:simpunit} says that, if $d_0^2 \zeta = a$, then $\zeta = s_1^1(a)$ or $\zeta = s_1^1(b)$. Thus, $n_{abb} = 1$, and $n_{aab} = 0$.

We haven't yet imposed the condition of Lemma \ref{lemma:simpassociativity}, but the cardinalities that remain unconstrained so far are $n_{bab}$ and $n_{bbb}$. It turns out, however, that the condition of Lemma \ref{lemma:simpassociativity} holds for all values of $n_{bab}$ and $n_{bbb}$. This can be shown by the long but straightforward process of considering each of the $16$ elements of $S\mathcal{X}$ and checking that the fibers of $\boundary_{02}$ and $\boundary_{13}$ have the same number of elements.

Thus, every choice of $n_{bab}$ and $n_{bbb}$ gives a monoid in $\Span$ on the $2$-element set $X_1 = \{a,b\}$. 
\end{example}

\subsection{From monoids to simplicial sets}

Suppose $(X, \eta, \mu)$ is a monoid in $\Span$, and let $E$ and $M$ be sets with maps $s_0^0: E \to X$ and $d_i^2: M \to X$, $i=0,1,2$, such that the spans
\begin{equation}\label{diag:monoidspan}
\begin{tikzcd}
    & E \arrow{dr}{s_0^0} \arrow{dl} & \\
    \bullet & & X
\end{tikzcd}
\;\;\;\;\;
\begin{tikzcd}
    & M \arrow{dr}{d_1^2} \arrow[swap]{dl}{(d_2^2,d_0^2)} & \\
    X \times X & & X
\end{tikzcd}
\end{equation}
represent $\eta$ and $\mu$, respectively. We will see that the maps $s_0^0$ and $d_i^2$ form part of the structure of a ($2$-truncated) simplicial set, where the set of $0$-simplices is $E$, the set of $1$-simplices is $X$, and the set of $2$-simplices is $M$.

The unit axiom can be expressed in terms of spans as follows. The composition $\mu \circ (\idx \times \eta)$ is represented by the diagram
    \[
    \begin{tikzcd}
    && (X \times E) * M \arrow{dl} \arrow{dr} && \\
    & X \times E \arrow[swap]{dl}{p_1} \arrow{dr}{1 \times s_0^0} && M \arrow[swap]{dl}{(d_2^2,d_0^2)} \arrow{dr}{d_1^2} & \\
    X && X \times X && X
    \end{tikzcd}
    \]
where $p_1: X \times E \to X$ is projection onto the first component, and where $(X \times E) * M$ is the pullback over $X \times X$. The equation $\mu \circ (\idx \times \eta) = \idx$ implies that there exists a bijection $X \isoto (X \times E) * M$ such that the diagram
    \begin{equation}\label{diag:unit}
    \begin{tikzcd}
    &[-10pt] & X \arrow[<->]{d} \arrow[bend left=30]{dddrr}{1} \arrow[bend right=30, swap]{dddll}{1} && \\
    && (X \times E) * M \arrow{dl} \arrow{dr} && \\
    & X \times E \arrow{dl}{p_1} \arrow{dr}{1 \times s_0^0} && M \arrow[swap]{dl}{(d_2^2,d_0^2)} \arrow[swap]{dr}{d_1^2} & \\
    X && X \times X && X
    \end{tikzcd}
    \end{equation}
commutes. Note that, because the outer maps in \eqref{diag:unit} are the identity, the bijection $X \isoto (X \times E) * M$ is unique.

The composition $X \isoto (X \times E) * M \to X \times E \to E$, where the last map is projection onto the second component, gives a map which we denote as $d_0^1: X \to E$. The composition $X \isoto (X \times E) * M \to M$ gives a map which we denote as $s_1^1: X \to M$. Using the maps we have just introduced, we obtain the commutative diagram
    \begin{equation}\label{diag:unit2}
    \begin{tikzcd}
    && X \arrow{dl}{(1, d_0^1)} \arrow[swap]{dr}{s_1^1} \arrow[bend left=30]{ddrr}{1} \arrow[bend right=30,swap]{ddll}{1}&& \\
    & X \times E \arrow{dr}{1 \times s_0^0} \arrow{dl}{p_1} && M \arrow[swap]{dl}{(d_2^2,d_0^2)} \arrow[swap]{dr}{d_1^2} &\\
   X && X \times X && X
    \end{tikzcd}
    \end{equation}
as an abbreviation of \eqref{diag:unit}.

Similarly, we can use the equation $\mu \circ (\eta \times \idx) = \idx$ to obtain maps $d_1^1: X \to E$ and $s_0^1: X \to M$, such that the diagram
    \begin{equation}\label{diag:unit3}
    \begin{tikzcd}
   & & X \arrow{dl}{(d_1^1,1)} \arrow[swap]{dr}{s_0^1} \arrow[bend left=30]{ddrr}{1} \arrow[swap,bend right=30]{ddll}{1} & &\\
     & E \times X \arrow{dr}{s_0^0 \times 1} \arrow{dl}{p_2} && M \arrow[swap]{dl}{(d_2^2,d_0^2)} \arrow[swap]{dr}{d_1^2} &\\
   X & & X \times X && X
    \end{tikzcd}
    \end{equation}
commutes.

We now turn to the associativity axiom, which can be expressed in terms of spans as follows. The diagram
    \begin{equation*}
    \begin{tikzcd}
    &[-15pt]& (X \times M) * M \arrow{dl} \arrow{dr} && \\
    & X \times M \arrow[swap]{dl}{1 \times (d_2^2,d_0^2)} \arrow{dr}{1 \times d_1^2} && M \arrow[swap]{dl}{(d_2^2,d_0^2)} \arrow{dr}{d_1^2} & \\
    X \times X \times X && X \times X && X
    \end{tikzcd}
    \end{equation*}
represents the composition $\mu \circ (\idx \times \mu)$. Using the identification
\begin{align*} 
(X \times M) * M &= \{(x,m,m') \in X \times M \times M \suchthat x = d_2^2 m', \; d_1^2 m = d_0^2 m'\} \\
&= \{(m,m') \suchthat d_1^2 m = d_0^2 m'\}\\
&= M \bitimes{d_1^2}{d_0^2} M,
\end{align*}
we obtain
    \begin{equation}\label{diag:ass1}
    \begin{tikzcd}
    &[-15pt]& M \bitimes{d_1^2}{d_0^2} M \arrow[swap]{dl}{(d_2^2 p_2, p_1)} \arrow{dr}{p_2} && \\
    & X \times M \arrow[swap]{dl}{1 \times (d_2^2,d_0^2)} \arrow{dr}{1 \times d_1^2} && M \arrow[swap]{dl}{(d_2^2,d_0^2)} \arrow{dr}{d_1^2} & \\
    X \times X \times X && X \times X && X
    \end{tikzcd}
    \end{equation}
as a representative of $\mu \circ (\idx \times \mu)$. Similarly, the diagram
    \begin{equation}\label{diag:ass2}
    \begin{tikzcd}
    &[-15pt]& M \bitimes{d_1^2}{d_2^2} M \arrow[swap]{dl}{(p_1, d_0^2 p_2)} \arrow{dr}{p_2} && \\
    & M \times X \arrow[swap]{dl}{(d_2^2,d_0^2) \times 1} \arrow{dr}{d_1^2 \times 1} && M \arrow[swap]{dl}{(d_2^2,d_0^2)} \arrow{dr}{d_1^2} & \\
    X \times X \times X && X \times X && X
    \end{tikzcd}
    \end{equation}
    represents the composition $\mu \circ (\mu \times \idx)$.
    
Associativity implies that there is a bijection $M \bitimes{d_1^2}{d_0^2} M \cong M \bitimes{d_1^2}{d_2^2} M$ giving an isomorphism of spans between \eqref{diag:ass1} and \eqref{diag:ass2}. It will be convenient to have a neutral model, so let $T$ be a set with maps $t_1: T \to X \times X \times X$ and $t_2: T \to X$ such that
\begin{equation}
    \begin{tikzcd}
        &[-15pt] T \arrow[swap]{dl}{t_1} \arrow{dr}{t_2}& \\
        X \times X \times X & & X
    \end{tikzcd}
\end{equation}
is isomorphic to \eqref{diag:ass1} and \eqref{diag:ass2}. We will use
\stringdiagram{
\tripleprod{0}{0}
}
to diagrammatically represent $T$.

\begin{thm}\label{thm:monoid2simp}
Suppose $(X, \eta, \mu)$ is a monoid in $\Span$, where $\eta$ and $\mu$ are represented by spans as in \eqref{diag:monoidspan}. Then the maps $d_i^k$ and $s_i^k$ in \eqref{diag:monoidspan}, \eqref{diag:unit2}, and \eqref{diag:unit3} are, respectively, the face and degeneracy maps for a $2$-truncated simplicial set $M \threearrows X \arrows E$.
\end{thm}    

\begin{proof}
    
The identities 
\begin{align}
    d_0^2 \circ s_1^1 &= s_0^0 \circ d_0^1, & 
    d_1^2 \circ s_1^1 &= d_2^2 \circ s_1^1 = 1, \label{eqn:ds1} \\
    d_2^2 \circ s_0^1 &= s_0^0 \circ d_1^1, &
    d_0^2 \circ s_0^1 &= d_1^2 \circ s_0^1 = 1.
    \label{eqn:ds2}
\end{align}
follow from the commutativity of \eqref{diag:unit2} and \eqref{diag:unit3}.

Now, consider the equation $\mu \circ (\idx \times \eta) \circ \eta = \mu \circ (\eta \times \eta) = \mu \circ (\eta \times \idx) \circ \eta$, diagrammatically shown as follows:

\stringdiagramlabel{
\begin{scope}
\unit{0}{2.5}
\identity{0}{.5}
\unit{2}{.5}
\multiplication{1}{-1.5}
\equals{3}{0}
\unit{4}{1}
\unit{6}{1}
\multiplication{5}{-1}
\equals{7}{0}
\unit{10}{2.5}
\unit{8}{.5}
\identity{10}{.5}
\multiplication{9}{-1.5}
\draw[gray,dotted] (-.5,1.5) rectangle (.5,3.5);
\draw[gray,dotted] (9.5,1.5) rectangle (10.5,3.5);
\end{scope}
}{diag:doubleid1}
Because of the unit axiom, we can see that the left and right sides of \eqref{diag:doubleid1} are postcompositions of $\eta$ (boxed) with the identity morphism on $X$. Additionally, all three diagrams in \eqref{diag:doubleid1} are related via a slide move. Together, these relationships allow us to obtain natural representatives of the left and right sides of \eqref{diag:doubleid1} as compositions of spans, shown in Figures \ref{fig:doubleid1} and \ref{fig:doubleid2}, that are isomorphic via the identity map $1:E \to E$. Because they both give the same representative of the middle of \eqref{diag:doubleid1} via the slide move, the maps $E \to E \times E$ and $E \to M$ are equal, so by comparing these maps we obtain the identities
\begin{align} \label{eqn:e}
    d_0^1 s_0^0 &= d_1^1 s_0^0, & s_1^1 s_0^0 &= s_0^1 s_0^0. 
\end{align}

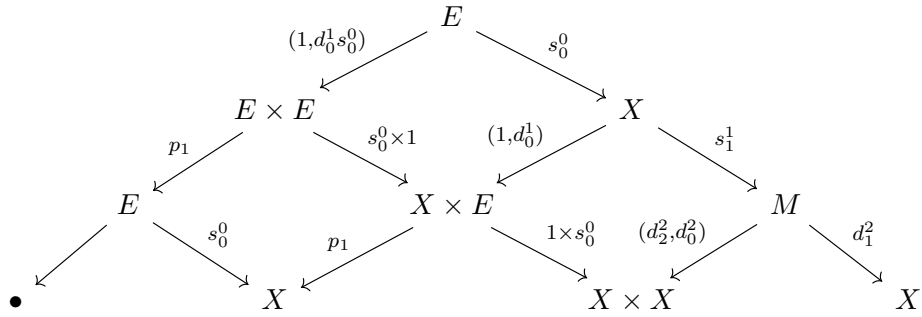
\begin{figure}[th]
\begin{tikzcd}
    &&& E \arrow[swap]{dl}{(1,d_0^1 s_0^0)} \arrow{dr}{s_0^0} &&& \\
    && E \times E \arrow[swap]{dl}{p_1} \arrow{dr}{s_0^0 \times 1} && X \arrow[swap]{dl}{(1,d_0^1)} \arrow{dr}{s_1^1} && \\
    & E \arrow{dl} \arrow{dr}{s_0^0} && X \times E \arrow[swap]{dl}{p_1} \arrow{dr}{1 \times s_0^0} && M \arrow[swap]{dl}{(d_2^2,d_0^2)} \arrow{dr}{d_1^2} & \\
    \bullet && X && X \times X && X
\end{tikzcd}
    \caption{Composition of spans representing the left side of \eqref{diag:doubleid1}.}
    \label{fig:doubleid1}
\end{figure}

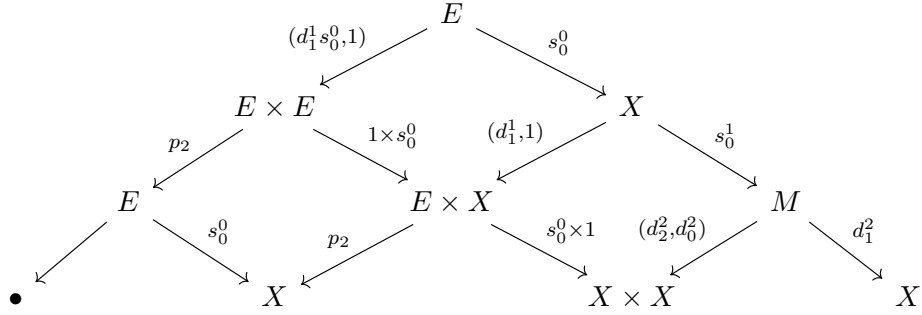
\begin{figure}[th]
\begin{tikzcd}
    &&& E \arrow[swap]{dl}{(d_1^1 s_0^0,1)} \arrow{dr}{s_0^0} &&& \\
    && E \times E \arrow[swap]{dl}{p_2} \arrow{dr}{1 \times s_0^0} && X \arrow[swap]{dl}{(d_1^1,1)} \arrow{dr}{s_0^1} && \\
    & E \arrow{dl} \arrow{dr}{s_0^0} && E \times X \arrow[swap]{dl}{p_2} \arrow{dr}{s_0^0 \times 1} && M \arrow[swap]{dl}{(d_2^2,d_0^2)} \arrow{dr}{d_1^2} & \\
    \bullet && X && X \times X && X
\end{tikzcd}
    \caption{Composition of spans representing the right side of \eqref{diag:doubleid1}.}
    \label{fig:doubleid2}
\end{figure}

Next, we consider the equation
\stringdiagramlabel{
\identity{-4}{2}
\unit{-2}{2}
\identity{-1}{2}
\multiplication{-3}{0}
\identity{-1}{0}
\multiplication{-2}{-2}
\equals{0}{0}
\identity{1}{1}
\unit{2}{1}
\identity{3}{1}
\tripleprod{2}{-1}
\equals{4}{0}
\identity{8}{2}
\unit{6}{2}
\identity{5}{2}
\multiplication{7}{0}
\identity{5}{0}
\multiplication{6}{-2}
\draw[gray,dotted] (-3.5,-3) rectangle (-.5,-.5);
\draw[gray,dotted] (4.5,-3) rectangle (7.5,-.5);
}{diag:midface}
which follows from associativity. The left and right sides of \eqref{diag:midface} are precompositions of $\mu$ (boxed) with the identity morphism on $X \times X$. As a result, we obtain the representatives in Figures \ref{fig:midface1} and \ref{fig:midface2} that are isomorphic via the identity map $1: M \to M$. Because they both should give the same representative of the middle of \eqref{diag:midface}, the maps $M \to X \times E \times X$ are equal, so we obtain the identity
\begin{equation}\label{eqn:faceface1}
    d_0^1 d_2^2 = d_1^1 d_0^2.
\end{equation}

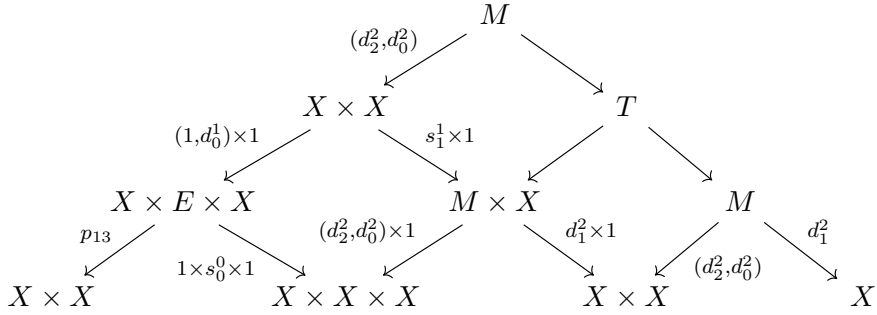
\begin{figure}[th]
    \begin{tikzcd}
    &[-30pt]&[-30pt]&[-25pt] M \arrow[swap]{dl}{(d_2^2,d_0^2)} \arrow{dr} &[-20pt]&[-15pt]&\\
    && X \times X \arrow[swap]{dl}{(1,d_0^1)\times 1} \arrow{dr}{s_1^1 \times 1} && T \arrow{dl} \arrow{dr} && \\
    & X \times E \times X \arrow[swap]{dl}{p_{13}} \arrow[swap]{dr}{1 \times s_0^0 \times 1} && M \times X \arrow[swap]{dl}{(d_2^2,d_0^2)\times 1} \arrow{dr}{d_1^2 \times 1} && M \arrow{dl}{(d_2^2,d_0^2)} \arrow{dr}{d_1^2} & \\
    X \times X &&X \times X \times X && X \times X && X
    \end{tikzcd}
    \caption{Composition of spans representing the left side of \eqref{diag:midface}.}
    \label{fig:midface1}
\end{figure}

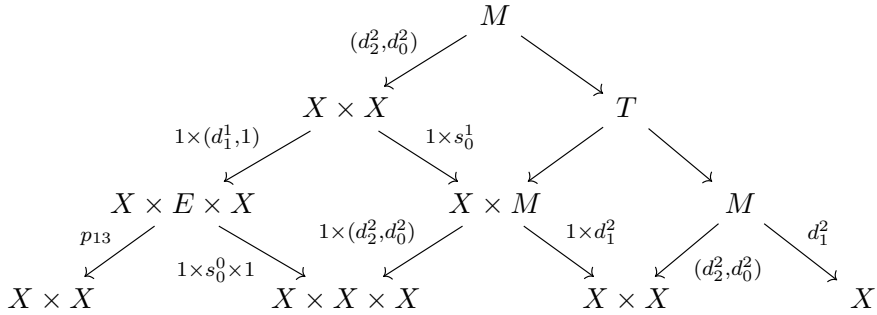
\begin{figure}[th]
    \begin{tikzcd}
    &[-30pt]&[-30pt]&[-25pt] M \arrow[swap]{dl}{(d_2^2,d_0^2)} \arrow{dr} &[-20pt]&[-15pt]&\\
    && X \times X \arrow[swap]{dl}{1 \times (d_1^1,1)} \arrow{dr}{1 \times s_0^1} && T \arrow{dl} \arrow{dr} && \\
    & X \times E \times X \arrow[swap]{dl}{p_{13}} \arrow[swap]{dr}{1 \times s_0^0 \times 1} && X \times M \arrow[swap]{dl}{1 \times (d_2^2,d_0^2)} \arrow{dr}{1 \times d_1^2} && M \arrow{dl}{(d_2^2,d_0^2)} \arrow{dr}{d_1^2} & \\
    X \times X &&X \times X \times X && X \times X && X
    \end{tikzcd}
    \caption{Composition of spans representing the right side of \eqref{diag:midface}.}
    \label{fig:midface2}
\end{figure}

Next, we consider the equation
\stringdiagramlabel{
\unit{-4}{2}
\identity{-2}{2}
\identity{-1}{2}
\multiplication{-3}{0}
\identity{-1}{0}
\multiplication{-2}{-2}
\equals{0}{0}
\unit{1}{1}
\identity{2}{1}
\identity{3}{1}
\tripleprod{2}{-1}
\equals{4}{0}
\identity{8}{2}
\identity{6}{2}
\unit{5}{2}
\multiplication{7}{0}
\identity{5}{0}
\multiplication{6}{-2}
\draw[gray,dotted] (-3.5,-3) rectangle (-.5,-.5);
}{diag:leftface}
which follows from associativity. The left side of \eqref{diag:leftface} is the precomposition of $\mu$ (boxed) with the identity map on $X \times X$, so it is naturally represented by the composition in Figure \ref{fig:leftface1}. To obtain a natural representative of the right side of \eqref{diag:leftface}, we first consider
    \stringdiagramlabel{
    \multiplication{1}{3}
    \unit{-1}{1}
    \identity{1}{1}
    \multiplication{0}{-1}
    \draw[gray,dotted] (-.5,2) rectangle (2.5,4.5);
    }{diag:leftface2}
which is a postcomposition of $\mu$ (boxed) with the identity map on $X$, and is naturally represented by the composition in Figure \ref{fig:leftface2}. We note that the map $M \to E \times M$ in Figure \ref{fig:leftface2} is determined by commutativity of the diagram.
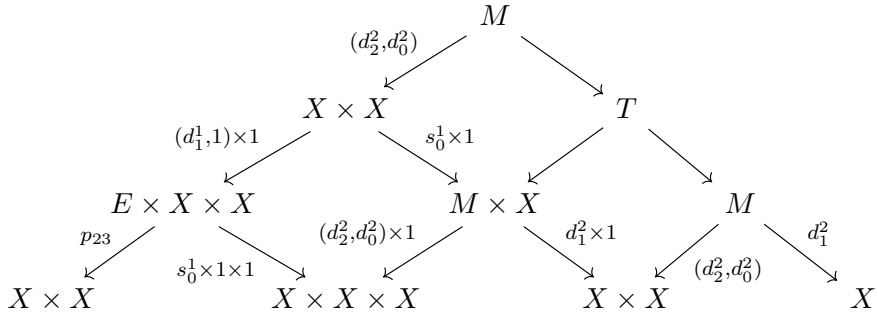
\begin{figure}[th]
    \begin{tikzcd}
    &[-30pt]&[-30pt]&[-25pt] M \arrow[swap]{dl}{(d_2^2,d_0^2)} \arrow{dr} &[-20pt]&[-15pt]&\\
    && X \times X \arrow[swap]{dl}{(d_1^1,1)\times 1} \arrow{dr}{s_0^1 \times 1} && T \arrow{dl} \arrow{dr} && \\
    & E \times X \times X \arrow[swap]{dl}{p_{23}} \arrow[swap]{dr}{s_0^1 \times 1 \times 1} && M \times X \arrow[swap]{dl}{(d_2^2,d_0^2)\times 1} \arrow{dr}{d_1^2 \times 1} && M \arrow{dl}{(d_2^2,d_0^2)} \arrow{dr}{d_1^2} & \\
    X \times X &&X \times X \times X && X \times X && X
    \end{tikzcd}
    \caption{Composition of spans representing the left side of \eqref{diag:leftface}.}
    \label{fig:leftface1}
\end{figure}
\begin{figure}[th]
    \begin{tikzcd}
    &[-10pt]&&[-15pt] M \arrow[swap]{dl}{(d_1^1 d_1^2,1)} \arrow{dr}{d^2_1} &[-15pt]&[-10pt]&\\
    && E \times M \arrow[swap]{dl}{p_2} \arrow{dr}{1 \times d_1^2} && X \arrow[swap]{dl}{(d_1^1,1)} \arrow{dr}{s_0^1} && \\
    & M \arrow[swap]{dl}{(d_2^2,d_0^2)} \arrow{dr}{d_1^2} && E \times X \arrow[swap]{dl}{p_2} \arrow{dr}{s_0^0 \times 1} && M \arrow{dl}{(d_2^2,d_0^2)} \arrow{dr}{d_1^2} & \\
    X \times X &&X && X \times X && X
    \end{tikzcd}
    \caption{Composition of spans representing \eqref{diag:leftface2}.}
    \label{fig:leftface2}
\end{figure}
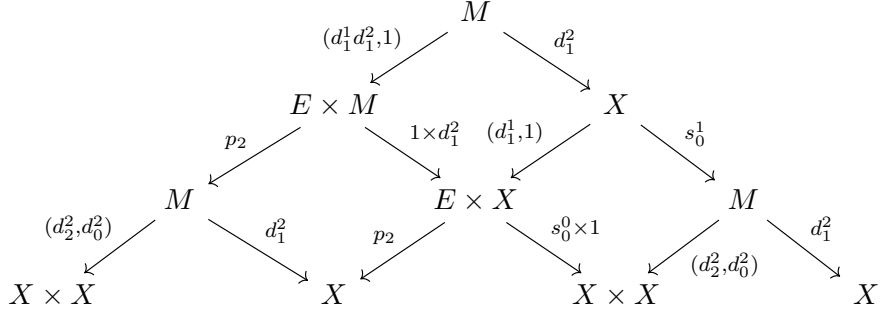
We then perform a slide move to obtain a natural representative of the right side of \eqref{diag:leftface}, shown in Figure \ref{fig:leftface3}.
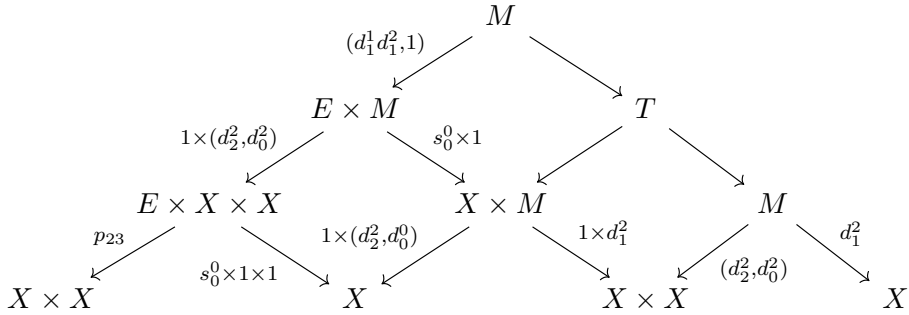
\begin{figure}[th]
    \begin{tikzcd}
    &[-20pt]&[-25pt]&[-15pt] M \arrow[swap]{dl}{(d_1^1 d_1^2,1)} \arrow{dr} &[-15pt]&[-10pt]&\\
    && E \times M \arrow[swap]{dl}{1 \times (d_2^2,d_0^2)} \arrow{dr}{s_0^0 \times 1} && T \arrow[swap]{dl} \arrow{dr} && \\
    & E \times X \times X \arrow[swap]{dl}{p_{23}} \arrow{dr}[swap]{s_0^0 \times 1 \times 1} && X \times M \arrow[swap]{dl}{1 \times (d_2^2, d_0^0)} \arrow{dr}{1 \times d_1^2} && M \arrow{dl}{(d_2^2,d_0^2)} \arrow{dr}{d_1^2} & \\
    X \times X &&X && X \times X && X
    \end{tikzcd}
    \caption{Composition of spans representing the right side of \eqref{diag:leftface}.}
    \label{fig:leftface3}
\end{figure}
Because the compositions in Figures \ref{fig:leftface1} and \ref{fig:leftface3} should both give the same representation of the middle of \eqref{diag:leftface}, the maps $M \to E \times X \times X$ are equal, so we obtain the identity
\begin{equation}\label{eqn:faceface2}
    d_1^1 d_2^2 = d_1^1 d_1^2.
\end{equation}
A similar analysis for the equation
\stringdiagramlabel{
\identity{-4}{2}
\identity{-2}{2}
\unit{-1}{2}
\multiplication{-3}{0}
\identity{-1}{0}
\multiplication{-2}{-2}
\equals{0}{0}
\identity{1}{1}
\identity{2}{1}
\unit{3}{1}
\tripleprod{2}{-1}
\equals{4}{0}
\unit{8}{2}
\identity{6}{2}
\identity{5}{2}
\multiplication{7}{0}
\identity{5}{0}
\multiplication{6}{-2}
}{diag:rightface}
yields the identity
\begin{equation}\label{eqn:faceface3}
    d_0^1 d_1^2 = d_0^1 d_0^2.
\end{equation}

Together, \eqref{eqn:ds1}, \eqref{eqn:ds2}, \eqref{eqn:e}, \eqref{eqn:faceface1}, \eqref{eqn:faceface2}, and \eqref{eqn:faceface3} are the identities necessary to show that $M \threearrows X \arrows E$ is a $2$-truncated simplicial set.
\end{proof}

One can readily verify that the constructions in Theorems \ref{thm:simp2monoid} and \ref{thm:monoid2simp} give a one-to-one correspondence, up to isomorphism, between monoids in $\Span$ and $2$-truncated simplicial sets satisfying the conditions in Lemmas \ref{lemma:simpunit} and \ref{lemma:simpassociativity}.

\subsection{Relation to 2-Segal sets}
\label{sec:2segal}

It was shown by Stern \cite{stern:2segal} that there is an $\infty$-categorical equivalence between pseudomonoids in the $(2,1)$-category of spans and $2$-Segal simplicial sets (see \cites{dyckerhoff-kapranov, gkt, boors, penney}). We sketch here the relationship between his result and the correspondence in Theorems \ref{thm:simp2monoid} and \ref{thm:monoid2simp}.

The pseudomonoids considered in \cite{stern:2segal} have more structure than the monoids in $\Span$ considered here, requiring specific isomorphisms of spans representing the unit and associativity axioms, as well as additional coherence conditions on these isomorphisms. Given such a pseudomonoid, one can obtain a monoid in $\Span$ by forgetting the isomorphisms. In the other direction, given a monoid in $\Span$, one could choose unitor and associator isomorphisms, but it may not be possible to make the choices so that the coherence conditions are satisfied.

Stern's correspondence, combined with ours, implies a similar relationship between $2$-Segal sets and simplicial sets satisfying the conditions in Theorem \ref{thm:simp2monoid}, which we can now see directly.

Suppose that $\mathcal{X}= X_\bullet$ is a $2$-Segal set. It was shown in \cite{fgkpw} that $X_\bullet$ automatically satisfies unital conditions, which in the lowest dimension say that the diagrams
\begin{equation*}
\begin{tikzcd}
    X_1 \arrow{r}{d_0^1} \arrow[swap]{d}{s_1^1} & X_0 \arrow{d}{s_0^0} \\
    X_2 \arrow{r}{d_0^2 }& X_1
\end{tikzcd}
\;\;\;\;
\begin{tikzcd}
    X_1 \arrow{r}{d_1^1} \arrow[swap]{d}{s_0^1} & X_0 \arrow{d}{s_0^0} \\
    X_2 \arrow{r}{d_2^2 }& X_1
\end{tikzcd}
\end{equation*}
are pullbacks. These conditions are equivalent to the conditions in Lemma \ref{lemma:simpunit}.

The lowest-dimension $2$-Segal conditions say that the maps
\begin{align*}
    (d_0^3,d_2^3): X_3 &\to X_2 \bitimes{d_1^2}{d_0^2} X_2,
 &   (d_1^3,d_3^3): X_3 &\to X_2 \bitimes{d_2^2}{d_1^2} X_2
\end{align*}
are bijections. We note that the codomains of these maps are exactly the taco sets $T_{02}\mathcal{X}$ and $T_{13}\mathcal{X}$. Composing the two bijections, we obtain a bijection $T_{02}\mathcal{X} \cong T_{13}\mathcal{X}$ satisfying the condition of Lemma \ref{lemma:simpassociativity}.

In the other direction, suppose that $X_\bullet$ is a $2$-truncated simplicial set satisfying the conditions of Theorem \ref{thm:simp2monoid}. Choose a bijection as in Lemma \ref{lemma:simpassociativity}, and set $X_3 \subseteq \Delta_3 \mathcal{X}$ to be the graph of the bijection. It is possible to choose the bijection carefully so that $X_3$ contains all degenerate $3$-simplices, thus giving a $3$-truncated simplicial set that can be extended to a $3$-coskeletal simplicial set, which we will still denote as $X_\bullet$. 

By construction $X_\bullet$ satisfies the lowest-dimension $2$-Segal conditions. However, it is not necessarily the case that $X_\bullet$ satisfies the higher-dimensional $2$-Segal conditions. These conditions constitute coherence conditions, which may or may not be possible to satisfy, on the bijection in Lemma \ref{lemma:simpassociativity}.

To summarize, if $\mathcal{X}$ is a $2$-Segal set, then the $2$-truncation of $\mathcal{X}$ satisfies the conditions of Theorem \ref{thm:simp2monoid}. We note that there is a wealth of examples of $2$-Segal sets in, e.g.\ \cites{dyckerhoff-kapranov,gkt,boors}, from which we can obtain many interesting examples of monoids in $\Span$. 

On the other hand, obtaining a $2$-Segal set from a simplicial set satisfying the conditions of Theorem \ref{thm:simp2monoid} requires a choice of bijection as in Lemma \ref{lemma:simpassociativity}, subject to nontrivial coherence conditions.

\section{Frobenius objects in \texorpdfstring{$\Span$}{Span}}
\label{sec:frob}

\subsection{Nondegeneracy}\label{sec:nondegeneracy}

Let $X$ be a set. A morphism $\alpha: X \times X \spanto \bullet$ in $\Span$ is called \emph{nondegenerate} if there exists a morphism $\beta: \bullet \spanto X \times X$ such that the \emph{snake identity}
\begin{equation}\label{eqn:nondegen}
    (\alpha \times \idx) \circ (\idx \times \beta) = (\idx \times \alpha) \circ (\beta \times \idx) = \idx
\end{equation}
holds. If we depict $\alpha$ and $\beta$, respectively, by
\stringdiagram{
\pairing{0}{0}
\copairing{4}{1}
}
then \eqref{eqn:nondegen} is given by
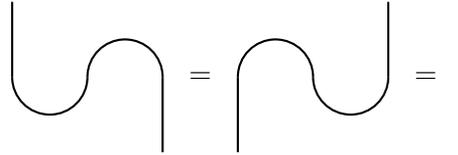
\stringdiagram{
\identity{-1}{2}
\copairing{2}{2}
\pairing{0}{0}
\identity{3}{0}
\equals{4}{1}
\copairing{6}{2}
\identity{9}{2}
\identity{5}{0}
\pairing{8}{0}
\equals{10}{1}
\identity{11}{2}
\identity{11}{0}
}
which explains where the name ``snake identity'' comes from.

Let $A$ be a set with maps $\alpha_1, \alpha_2: A \to X$, so that the span
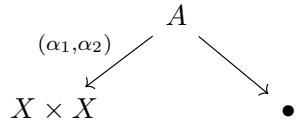
\begin{equation*}
    \begin{tikzcd}
        & [-10pt] A \arrow[swap]{dl}{(\alpha_1, \alpha_2)} \arrow{dr}& \\
        X \times X & & \bullet
    \end{tikzcd}
\end{equation*}
represents a morphism $\alpha: X \times X \spanto \bullet$. The following can be deduced from Proposition \ref{prop:spaniso} and the well-known fact that $\Span$ is a rigid category, but we include a more explicit proof here.

\begin{prop}\label{prop:nondegen}
$\alpha$ is nondegenerate if and only if $\alpha_1$ and $\alpha_2$ are bijections.
\end{prop}

\begin{proof}
Consider a span 
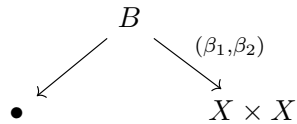
\begin{equation*}
    \begin{tikzcd}
        & B \arrow{dr}{(\beta_1, \beta_2)} \arrow{dl}& [-10pt]\\
        \bullet && X \times X
    \end{tikzcd}
\end{equation*}
where $B$ is a set with maps $\beta_1, \beta_2: B \to X$, and let $\beta: \bullet \spanto X \times X$ denote the corresponding morphism in $\Span$. Then the composition of spans
\begin{equation}\label{diag:nondegen1}
    \begin{tikzcd}
& & [-15pt]   (X \times B) * (A \times X) \arrow{dl} \arrow{dr} &[-15pt]&\\
& X \times B \arrow[swap]{dl}{p_1} \arrow[swap]{dr}{1 \times(\beta_1, \beta_2)} && A \times X \arrow{dl}{(\alpha_1,\alpha_2) \times 1} \arrow{dr}{p_2} &\\
X && X \times X \times X && X        
    \end{tikzcd}
\end{equation}
represents $(\alpha \times \idx) \circ (\idx \times \beta)$. Using the identification
\begin{align*}
    &(X \times B) * (A \times X) \\
    &= \{(x,b,a,x') \in X \times B \times A \times X \suchthat x = \alpha_1(a), \beta_1(b) = \alpha_2(a), \beta_2(b) = x'\} \\
    &= A \bitimes{\alpha_2}{\beta_1} B,
\end{align*}
we can rewrite \eqref{diag:nondegen1} as
\begin{equation*}
    \begin{tikzcd}
& & [-15pt]   A \bitimes{\alpha_2}{\beta_1} B \arrow[swap]{dl}{\alpha_1 \times 1} \arrow{dr}{1 \times \beta_2} &[-15pt]&\\
& X \times B \arrow[swap]{dl}{p_1} \arrow[swap]{dr}{1 \times(\beta_1, \beta_2)} && A \times X \arrow{dl}{(\alpha_1,\alpha_2) \times 1} \arrow{dr}{p_2} &\\
X && X \times X \times X && X        
    \end{tikzcd}
\end{equation*}
which is isomorphic to
\begin{equation}\label{diag:nondegen3}
    \begin{tikzcd}
& & [-15pt]   A \bitimes{\alpha_2}{\beta_1} B \arrow[swap]{dl}{p_1} \arrow{dr}{p_2} &[-15pt]&\\
& A \arrow[swap]{dl}{\alpha_!} \arrow{dr}{\alpha_2} && B \arrow{dl}[swap]{\beta_1} \arrow{dr}{\beta_2} &\\
X && X && X        
    \end{tikzcd}
\end{equation}
via the identity map on $A \bitimes{\alpha_2}{\beta_1} B$. Similarly, the span
\begin{equation}\label{diag:nondegen4}
    \begin{tikzcd}
& & [-15pt]   B \bitimes{\beta_2}{\alpha_1} A \arrow[swap]{dl}{p_1} \arrow{dr}{p_2} &[-15pt]&\\
& B \arrow[swap]{dl}{\beta_!} \arrow{dr}{\beta_2} && A \arrow{dl}[swap]{\alpha_1} \arrow{dr}{\alpha_2} &\\
X && X && X        
    \end{tikzcd}
\end{equation}
represents the composition $(\idx \times \alpha) \circ (\beta \times \idx)$.

Thus we see that $\alpha$ is nondegenerate if and only if there exists $(B, \beta_1, \beta_2)$ such that the spans \eqref{diag:nondegen3} and \eqref{diag:nondegen4} are isomorphic to the identity on $X$. This is the case if and only if the morphism $\alpha^\flat: X \spanto X$ represented by the span $(A, \alpha_1, \alpha_2)$ is an isomorphism. The result then follows from Proposition \ref{prop:spaniso}.
\end{proof}

An immediate consequence of Proposition \ref{prop:nondegen} is the following.

\begin{cor}
A morphism $\alpha: X \times X \spanto \bullet$ is nondegenerate if and only if it can be represented by a span of the form
\begin{equation}\label{diag:alpha}
\begin{tikzcd}
    & X \arrow[swap]{dl}{(1, \hat{\alpha})} \arrow{dr} & \\
    X \times X & & \bullet
\end{tikzcd}
\end{equation}  
where $\hat{\alpha}: X \to X$ is a bijection. This representation is unique.
\end{cor}

\begin{remark}
When $\alpha$ is represented by \eqref{diag:alpha}, then the span
\begin{equation}\label{diag:beta}
\begin{tikzcd}
    & X \arrow{dr}{(\hat{\alpha},1)} \arrow{dl} & \\
    \bullet && X \times X
\end{tikzcd}
\end{equation} 
represents the corresponding morphism $\beta: \bullet \spanto X \times X$.
\end{remark}

\subsection{Frobenius objects}

\begin{definition}
A \emph{Frobenius object} in $\Span$ is a monoid $(X, \eta, \mu)$ in $\Span$ that is equipped with a morphism $\varepsilon: X \spanto \bullet$ (\emph{counit}), such that $\varepsilon \circ \mu$ is nondegenerate.
\end{definition}

Suppose that $X$ is a Frobenius object in $\Span$, and let $\alpha = \varepsilon \circ \mu$. Since $\alpha$ is nondegenerate, there is a unique bijection $\hat{\alpha}: X \to X$ such that $\alpha$ is represented by \eqref{diag:alpha}.
The counit can then be recovered from $\hat{\alpha}$, as follows.
\begin{lemma}\label{lemma:counit}
If the unit morphism $\eta$ is represented as in \eqref{diag:monoidspan}, then
\begin{equation}\label{diag:counit}
\begin{tikzcd}
    & E \arrow[swap]{dl}{\hat{\alpha}\circ s_0^0} \arrow{dr} & \\
    X & & \bullet
\end{tikzcd}
\end{equation}
represents the counit $\varepsilon$.
\end{lemma}

\begin{proof}
It follows from the unit axiom that $\varepsilon = \varepsilon \circ \mu \circ (\eta \times \idx) = \alpha \circ (\eta \times \idx)$. Thus, the composition of spans
\begin{equation}\label{diag:counit1}
    \begin{tikzcd}
        && [-15pt] (E \times X) * X \arrow{dl} \arrow{dr} &[-15pt]& \\
        & E \times X \arrow[swap]{dl}{p_2} \arrow{dr}{s_0^0 \times 1} && X \arrow[swap]{dl}{(1,\hat{\alpha})} \arrow{dr} & \\
        X && X \times X && \bullet
    \end{tikzcd}
\end{equation}
represents $\varepsilon$. Using the identification
\begin{align*}
    (E \times X) * X &= \{(e,x,x') \in E \times X \times X \suchthat s_0^0(e) = x', x = \hat{\alpha}(x')\} \\
    &= E,
\end{align*}
we can rewrite \eqref{diag:counit1} as
\begin{equation*}
    \begin{tikzcd}
        && [-15pt] E \arrow[swap]{dl}{(1, \hat{\alpha} \circ s_0^0)} \arrow{dr}{s_0^0} &[-15pt]& \\
        & E \times X \arrow[swap]{dl}{p_2} \arrow{dr}{s_0^0 \times 1} && X \arrow[swap]{dl}{(1,\hat{\alpha})} \arrow{dr} & \\
        X && X \times X && \bullet
    \end{tikzcd}
\end{equation*}
which gives the result.
\end{proof}

Now, suppose that $(X, \eta, \mu)$ is a monoid in $\Span$, and that $\hat{\alpha}:X \to X$ is a bijection. We can then define a morphism $\varepsilon: X \spanto \bullet$ via \eqref{diag:counit}. As a result of the correspondence in Lemma \ref{lemma:counit}, nondegeneracy of $\varepsilon \circ \mu$ is equivalent to the condition $\alpha = \varepsilon \circ \mu$. This gives us an alternative characterization of Frobenius structures.
\begin{prop}\label{prop:frob}
Suppose $(X, \eta, \mu)$ is a monoid in $\Span$ equipped with a bijection $\hat{\alpha}:X \to X$. Let $\alpha: X \times X \spanto \bullet$ and $\varepsilon: X \spanto \bullet$ be defined as in \eqref{diag:alpha} and \eqref{diag:counit}, respectively. Then $(X, \eta, \mu, \varepsilon)$ is a Frobenius object in $\Span$ if and only if $\alpha = \varepsilon \circ \mu$.
\end{prop}
We remark that the result of Proposition \ref{prop:frob} holds more generally in any rigid monoidal category and is known by some people.

The condition in Proposition \ref{prop:frob} makes it relatively straightforward to understand Frobenius structures in terms of the simplicial set perspective, since we can view $\hat{\alpha}$ as a bijection of the set of $1$-simplices. 

\begin{thm}\label{thm:frobenius}
Let $X_\bullet$ be a $2$-truncated simplicial set satisfying the conditions in Lemmas \ref{lemma:simpunit} and \ref{lemma:simpassociativity}, and let $\hat{\alpha}:X_1 \to X_1$ be a bijection. Then the corresponding monoid in $\Span$, as given by Theorem \ref{thm:simp2monoid}, together with $\varepsilon$ given by \eqref{diag:counit}, is a Frobenius object in $\Span$ if and only if there exists a map $\gamma:X_1 \to X_2$ such that
\begin{enumerate}
    \item  $d_0^2 \gamma(x) = \hat{\alpha}(x)$, $d_1^2 \gamma(x) \in \hat{\alpha}(s_0^0(X_0))$, and $d_2^2 \gamma(x) = x$ for all $x \in X_1$, and 
    \item if $\zeta \in X_2$ is such that $d_1^2 \zeta \in \hat{\alpha}(s_0^0(X_0))$, then $\zeta \in \gamma(X_1)$.
\end{enumerate}
\end{thm}

\begin{proof}
From Proposition \ref{prop:frob}, we have that $X_1$, with the given data, is a Frobenius object in $\Span$ if and only if $\alpha = \varepsilon \circ \mu$, where $\alpha$ and $\varepsilon$ are given by \eqref{diag:alpha} and \eqref{diag:counit}, respectively. Thus we consider the composition
\begin{equation}\label{diag:alphafrob}
    \begin{tikzcd}
        && X_2 * X_0 \arrow{dl} \arrow{dr} && \\
        & X_2 \arrow[swap]{dl}{(d_2^2,d_0^2)} \arrow{dr}{d_1^2} && X_0 \arrow[swap]{dl}{\hat{\alpha}\circ s_0^0} \arrow{dr} \\
        X_1 \times X_1 && X_1 && \bullet
    \end{tikzcd}
\end{equation}
which represents $\varepsilon \circ \mu$. We can rewrite the pullback in \eqref{diag:alphafrob} using the following identification:
\begin{equation}\label{eqn:epsilonmu}
\begin{split}
    X_2 * X_0 &= \{(\zeta, u) \in X_2 \times X_0 \suchthat d_1^2(\zeta) = \hat{\alpha} \circ s_0^0(u)\} \\
    &= \{ \zeta \in X_2 \suchthat d_1^2(\zeta) \in \hat{\alpha}(s_0^0(X_0))\}.
\end{split}
\end{equation}
Here, we have used the fact that $\hat{\alpha}$ and $s_0^0$ are both injective, so the $u$, if it exists, is unique.

The equation $\alpha = \varepsilon \circ \mu$ holds if and only if there exists an isomorphism of spans from \eqref{diag:alpha} to \eqref{diag:alphafrob}. Using \eqref{eqn:epsilonmu}, we see that such an isomorphism is given by a map $\gamma: X_1 \to X_2$ such that $d_1^2(\gamma(x)) \in \hat{\alpha}(s_0^0(X_0))$ for all $x \in X_1$. Compatibility with the spans adds the requirements that $d_2^2(\gamma(x)) = x$ and $d_0^2(\gamma(x)) = \hat{\alpha}(x)$. These conditions already imply that $\gamma$ is injective, so for $\gamma$ to be an isomorphism we only need to add the surjectivity condition (2).
\end{proof}

Intuitively, we can think of Theorem \ref{thm:frobenius} as saying that a Frobenius structure gives an additional degeneracy-type map $\gamma$, relative to the bijection $\hat{\alpha}$. Condition (1) expresses compatibility conditions between $\gamma$ and the face maps, and condition (2) is analogous to the conditions in Lemma \ref{lemma:simpunit}. The following example is illustrative.

\begin{example}\label{ex:groupoid}
Let $G_1 \arrows G_0$ be a groupoid, and let $G_\bullet$ be the nerve; see Example \ref{ex:cat}. Let $\hat{\alpha}: G_1 \to G_1$ be given by $\hat{\alpha}(x) = x^{-1}$. Then there is a unique $\gamma: G_1 \to G_2$ satisfying condition (1) in Theorem \ref{thm:frobenius}, given by $\gamma(x) = (x,x^{-1})$. Condition (2) then says that, if $x,y \in G_1$ are such that $xy$ is an identity morphism, then $y = x^{-1}$, which is true for a groupoid.

It is worth noting that the above Frobenius structure is not the only possibility for a groupoid. The most general possibility is as follows. Let $\sigma: G_0 \to G_1$ be a section of the target map, and set
\[ \hat{\alpha}(x) = x^{-1}\sigma(1_{t(x)}).\]
Then the map $\gamma(x) = (x, \hat{\alpha}(x))$ satisfies the conditions of Theorem \ref{thm:frobenius}.
\end{example}

\subsection{Comultiplication}

If $X$ is a Frobenius object in $\Span$, then it has a naturally-induced \emph{comultiplication} $\delta: X \spanto X \times X$, defined as $\delta = (\idx \times \mu) \circ (\beta \times \idx)$; see Figure \ref{fig:comult}. General diagrammatic arguments prove that $\delta$ is counital (with counit $\varepsilon$), coassociative, and satisfies the \emph{Frobenius equation} $(\mu \times \idx) \circ (\idx \times \delta) = \delta \circ \mu = (\idx \times \mu) \circ (\delta \circ \idx)$.
\begin{figure}[th]
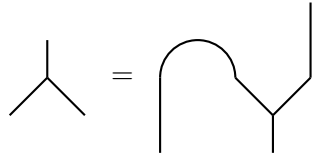

\stringdiagram{
\comultiplication{0}{0}
\equals{2}{0}
\copairing{4}{1}
\identity{7}{1}
\identity{3}{-1}
\multiplication{6}{-1}
}
    \caption{Definition of comultiplication via string diagrams.}
    \label{fig:comult}
\end{figure}

\begin{prop}
Let $\mu$ and $\alpha$ be represented as in \eqref{diag:monoidspan} and \eqref{diag:alpha}. Then
\[
\begin{tikzcd}
    & M \arrow[swap]{dl}{d_0^2} \arrow{dr}{(\hat{\alpha} \circ d_2^2, d_1^2)}& \\
    X && X \times X
\end{tikzcd}
\]
represents the comultiplication $\delta$.
\end{prop}
\begin{proof}
The composition of spans
\[
\begin{tikzcd}
    &&[-20pt] M \arrow[swap]{dl}{(d_2^2,d_0^2)} \arrow{dr}{(\hat{\alpha}\circ d_2^2, 1)} &[-20pt]&[-15pt] \\
    & X \times X \arrow[swap]{dl}{p_2} \arrow[swap]{dr}{(\hat{\alpha},1)\times 1} && X \times M \arrow[swap]{dl}{1 \times (d_2^2,d_0^2)} \arrow{dr}{1 \times d_1^2} & \\
    X && X \times X \times X && X \times X
\end{tikzcd}
\]
represents $(\idx \times \mu) \circ (\beta \times \idx) = \delta$.
\end{proof}

\subsection{Commutativity and TQFT}

Let $X$ be a set, and let $\hat{\tau}: X \times X \to X \times X$ be the ``twist map'' that exchanges the components: $\hat{\tau}(x,x') = (x',x)$. The span
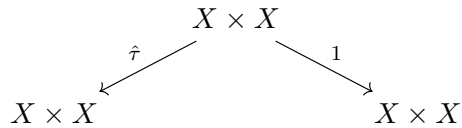
\[
\begin{tikzcd}
    & X \times X \arrow[swap]{dl}{\hat{\tau}} \arrow{dr}{1}& \\
    X \times X && X \times X
\end{tikzcd}
\]
gives an induced ``twist morphism'' $\tau: X \times X \spanto X \times X$ in $\Span$.

A monoid $X$ in $\Span$ is \emph{commutative} if $\mu \circ \tau = \mu$. A straightforward calculation shows that, if $\mu$ is represented as in \eqref{diag:monoidspan}, then $X$ is commutative if and only if there exists a bijection $\theta: M \to M$ such that $d_1^2 \circ \theta = d_1^2$, $d_2^2 \circ \theta = d_0^2$, and $d_0^2 \circ \theta = d_2^2$. 

From the general theory of TQFT (e.g. \cite{kock-book}), we know that a commutative Frobenius object in $\Span$ is equivalent to a symmetric monoidal functor from the oriented two-dimensional cobordism category to $\Span$. From such a functor, we can obtain invariants of closed oriented surfaces as the partition function of the TQFT. Specifically, for the genus $g$ surface $\Sigma_g$, the partition function is the morphism
\begin{equation}\label{eqn:partition}
 Z(\Sigma_g) = \varepsilon \circ (\mu \circ \delta)^g \circ \eta: \bullet \spanto \bullet.
\end{equation}
Such a morphism is equivalent to an isomorphism class of sets and can therefore be identified with a cardinality. If we're lucky, the cardinality will be finite, in which case $Z(\Sigma_g)$ can be viewed as a natural number.

A convenient way to compute $Z(\Sigma_g)$ is to count the number of trajectories in \eqref{eqn:partition}; see Section \ref{sec:spanww}.

\subsection{Symmetric Frobenius objects}\label{sec:symmetric}

A Frobenius object is called \emph{symmetric} if $\alpha \circ \tau = \alpha$. 

\begin{prop}
A Frobenius object in $\Span$ is symmetric if and only if $\hat{\alpha}^2 = 1$.
\end{prop}
\begin{proof}
The composition of spans
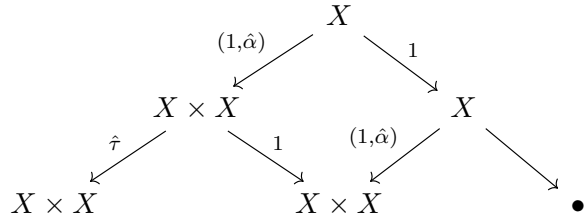
\[
\begin{tikzcd}
    &[-15pt]&[-15pt] X \arrow[swap]{dl}{(1,\hat{\alpha})} \arrow{dr}{1} &[-10pt]& \\
    & X \times X \arrow[swap]{dl}{\hat{\tau}} \arrow{dr}{1} && X \arrow[swap]{dl}{(1,\hat{\alpha})} \arrow{dr} & \\
    X\times X && X \times X && \bullet
\end{tikzcd}
\]
represents $\alpha \circ \tau$. The equation $\alpha \circ \tau = \alpha$ holds if and only if there is a bijection $f:X \to X$ such that $(1, \hat{\alpha}) \circ f = \hat{\tau} \circ (1, \hat{\alpha})$. Expressing the latter condition in components, we obtain the equations $\hat{\alpha} = f$ and $\hat{\alpha} \circ f = 1$, from which the result follows.
\end{proof}

\begin{example}
Let $G$ be a group, and let $\omega \in G$ be an arbitrary (fixed) element of $G$. Then $G$ has the structure of a Frobenius object in $\Span$, with $\hat{\alpha}(x) = x^{-1}\omega$ (see Example \ref{ex:groupoid}). Since $\hat{\alpha}^2(x) = \omega^{-1} x \omega$, we see that this Frobenius structure is symmetric if and only if $\omega$ is in the center of $G$.
\end{example}

The second main result in Stern's paper \cite{stern:2segal} (building on the result discussed in Section \ref{sec:2segal}) is an $\infty$-categorical equivalence between Calabi-Yau algebra objects in the $(2,1)$-category of spans and $2$-Segal cyclic sets. This result relates to the result of Theorem \ref{thm:frobenius} in a way that parallels the discussion in Section \ref{sec:2segal}. Roughly, given a Calabi-Yau algebra object, one can obtain a symmetric Frobenius object by forgetting the higher categorical data. On the other hand, given a symmetric Frobenius object, one could try to choose the required higher data, but there are nontrivial coherence conditions that would need to be satisfied in order to obtain a Calabi-Yau algebra object.

As in Section \ref{sec:2segal}, there is a corresponding relationship on the simplicial set side. A cyclic structure on a $2$-Segal set includes an involution on $X_1$ that satisfies the conditions of Theorem \ref{thm:frobenius}, but the other direction requires extending the action of $\hat{\alpha}$ on $X_1$ to higher-dimensional simplices in a coherent way.

Finally, we point out that the result of Theorem \ref{thm:frobenius} includes non-symmetric Frobenius objects, and this suggests the existence of a higher categorical analogue that extends Stern's result.

\section{Examples}\label{sec:examples}

In this section, we look at some examples of commutative Frobenius objects in $\Span$, and we compute the associated topological invariants for closed surfaces.

\begin{example}\label{ex:abeliangp}
Let $G$ be a finite abelian group, and fix an element $\omega \in G$. Then, as a special case of Example \ref{ex:groupoid}, we obtain a commutative Frobenius object in $\Span$ associated to the nerve of $G$, where $\hat{\alpha}(x) = x^{-1}\omega$. In this case, the unit, multiplication, counit, and coumultiplication are given by the following compositions:
\begin{align*}
    \eta&: \bullet \tolabel{e} G,\\
    \mu&: G \times G \tolabel{m} G,\\
    \varepsilon&: G \fromlabel{\omega} \bullet,\\
    \delta&: G \fromlabel{p_2} G \times G \tolabel{(\hat{\alpha}\circ p_1, m)} G \times G.
\end{align*}
To compute \eqref{eqn:partition}, we first consider the trajectories from $x$ to $y$ in $\mu \circ \delta: G \spanto G$. Using the above formulas, we can see that such a trajectory only exists when $y = \omega x$ and is of the form
\[ x \mapsfrom (z,x) \mapsto (z^{-1}\omega, zx) \mapsto \omega x = y.\]
Thus the number of trajectories from $x$ to $y$ in $\mu \circ \delta$ is $|G|$ if $y = \omega x$ and $0$ otherwise.

By iteration, we have that, for any $g \geq 0$, the number of trajectories from $x$ to $y$ in $(\mu \circ \delta)^g$ is $|G|^g$ if $y = \omega^g x$ and $0$ otherwise. Incorporating the unit and counit in \eqref{eqn:partition} has the effect of setting $x=e$ and $y=\omega$. Thus the partition function is
\[ Z(\Sigma_g) = 
\begin{cases}
|G|^g & \mbox{if } \omega^g = \omega,\\
0 & \mbox{otherwise}.
\end{cases}
\]
In particular, if we set $\omega = e$, then $Z(S^2) = 0$ and $Z(\Sigma_g) = |G|^g$ for $g \geq 1$.
\end{example}

\begin{example}\label{ex:twoelement2}
Recall the family of monoids in $\Span$ described in Example \ref{ex:twoelement}, which are all commutative. We now consider the compatible Frobenius structures. Since $X = \{a,b\}$ has two elements, there are two possibilities for $\hat{\alpha}$.

We first consider the case where $\hat{\alpha}$ is the identity map. In this case, condition (1) in Theorem \ref{thm:frobenius} says that $\gamma(a) \in Y_{aaa}$ and $\gamma(b) \in Y_{bab}$. Condition (2) in Theorem \ref{thm:frobenius} implies that there are no other elements of $Y_{bab}$, so $n_{bab} = 1$. However, $n_{bbb}$ remains unconstrained, leaving an infinite family of Frobenius objects in $\Span$. For simplicity of notation, we write $n = n_{bbb}$.

In this case, the unit and counit both give one trajectory between $\bullet$ and $a$. The numbers of trajectories for the multiplication and comultiplication are given in Figure \ref{fig:twoelement}.
\begin{figure}[th]
\begin{tabular}{c|c|c}
$\mu$    & $a$   & $b$ \\ \hline
$(a,a)$  & $1$   & $0$ \\ \hline
$(a,b)$ & $0$   & $1$ \\ \hline
$(b,a)$ & $0$   & $1$ \\ \hline
$(b,b)$ & $1$   & $n$
\end{tabular}
\hspace{.5in}
\begin{tabular}{c|c|c|c|c}
$\delta$ & $(a,a)$ & $(a,b)$ & $(b,a)$ & $(b,b)$ \\ \hline
$a$ & $1$   & $0$   & $0$   & $1$ \\ \hline
$b$ & $0$   & $0$   & $1$   & $n$
\end{tabular}
    \caption{Numbers of trajectories in $\mu$ and $\delta$ for Example \ref{ex:twoelement2} when $\hat{\alpha} = 1$.}
    \label{fig:twoelement}
\end{figure}
From these tables, we can calculate the numbers of trajectories in $\mu \circ \delta$ as the entries in the matrix
\[ A = \mat{2 & n \\ n & n^2+1}.\]
For example, there are $2$ trajectories from $a$ to $a$ (one passing through $(a,a)$ and one passing through $(b,b)$) and $n$ trajectories from $a$ to $b$ (all passing through $(b,b)$). Taking powers of $A$ gives us the numbers of trajectories in $(\mu \circ \delta)^g$ for any $g \geq 0$, so we can use matrix algebra to calculate the partition function:
\begin{align*}
    Z(\Sigma_g) &= \mat{1&0}A^g\mat{1\\0} \\
    &= \mat{1&0} \left(\frac{1}{1+n^2}\mat{1&-n\\n&1}\mat{n^2+2&0\\0&1}\mat{1&n\\-n&1}\right)^g \mat{1\\0}\\
    &= \frac{1}{1+n^2}\mat{1&-n}\mat{(n^2+2)^g&0\\0&1}\mat{1\\-n}\\
    &= \frac{(n^2+2)^g + n^2}{1+n^2}.
\end{align*}

Now we consider the case where $\hat{\alpha}$ is the nontrivial automorphism of $X = \{a,b\}$. In this case, condition (1) in Theorem \ref{thm:frobenius} says that $\gamma(a) \in Y_{bba}$ and $\gamma(b) \in Y_{abb}$. Condition (2) in Theorem \ref{thm:frobenius} implies that there are no elements in $Y_{bbb}$, so $n_{bbb} = 0$. However, $n_{bab}$ remains unconstrained, leaving an infinite family of Frobenius objects in $\Span$. Write $m = n_{bab}$. By calculations similar to the previous case, we obtain the matrix
\[ B = \mat{0&2m\\2&0}\]
with the numbers of trajectories in $\mu \circ \delta$. Then the partition function is
\[ Z(\Sigma_g) = \mat{0 & 1} B^g \mat{1\\0} = \begin{cases}
2^g m^{(g-1)/2} & \mbox{if $g$ odd},\\
0 & \mbox{if $g$ even}.
\end{cases}.\]
\end{example}

\section{Frobenius algebras via \texorpdfstring{$\Span$}{Span}} \label{sec:vect}

Let $\kk$ be a field, and let $\vect_\kk$ denote the category of vector spaces over $\kk$. Let $\FSpan$ denote the subcategory of $\Span$ consisting of finite sets and finite spans.

There is a functor from $\FSpan$ to $\vect_\kk$, which on objects takes a finite set $X$ to the space $\kk[X]$ of $\kk$-valued functions on $X$. Given a finite span $(A, f_1, f_2)$ from $X$ to $Y$, the induced map $\kk[X] \to \kk[Y]$ is obtained by pulling back along $f_1$ and then fiberwise summing over $f_2$. This functor is described, for example, in \cites{dyckerhoff-kapranov,morton:two} and is a simplified version of the degroupoidification functor described in \cite{bhw:groupoidification}.

The functor $\FSpan \to \vect_\kk$ is symmetric monoidal, taking the Cartesian product in $\FSpan$ to the tensor product in $\vect_\kk$. Thus it induces a map taking Frobenius objects in $\FSpan$ to Frobenius algebras over $\kk$, with commutative Frobenius objects going to commutative Frobenius algebras.

Suppose that $X_\bullet$ is a finite $2$-truncated simplicial set satisfying the conditions of Theorem \ref{thm:simp2monoid}, equipped with a bijection $\hat{\alpha}: X_1 \to X_1$, satisfying the conditions of Theorem \ref{thm:frobenius}. Then $X_\bullet$ corresponds to a Frobenius object in $\FSpan$. The induced Frobenius algebra is $\kk[X_1]$, where the multiplication is given by a convolution product
\[ \mu(\varphi, \psi)(x) = \sum_{\substack{\zeta \in X_2,\\ x=d_1^2(\zeta)}} \varphi(d_0^2(\zeta)) \psi(d_2^2(\zeta))\]
for $\varphi, \psi \in \kk[X_1]$. The identity element $\mathbbm{1} \in \kk[X_1]$ is given by
\[ \mathbbm{1}(x) = \begin{cases}
1 & \mbox{ if $x \in d_0^0(X_0)$,} \\
0 & \mbox{ otherwise,}
\end{cases}
\]
and the counit $\varepsilon: \kk[X_1] \to \kk$ is given by
\[ \varepsilon(\varphi) = \sum_{u \in X_0} \varphi(\hat{\alpha} \circ d_0^0(u)).\]
We note that the conditions in Theorems \ref{thm:simp2monoid} and \ref{thm:frobenius} could be interpreted (in the finite case) as necessary and sufficient conditions for the above data to satisfy the axioms of a Frobenius algebra.

\begin{example}\label{ex:groupoidvect}
Let $G_1 \arrows G_0$ be a finite groupoid, and let $X = G_1$ be the corresponding Frobenius object in $\FSpan$ (see Example \ref{ex:groupoid}). Then the induced Frobenius algebra is $\kk[G_1]$, where the multiplication is the convolution product, given by
\[ \mu(\varphi, \psi)(x) = \sum_{\substack{(y,z) \in G_2,\\ x = yz}} \varphi(z) \psi(y)\]
for $\varphi, \psi \in \kk[G_1]$. If the counit is given as in Example \ref{ex:groupoid} by a section $\sigma: G_0 \to G_1$ of the target map, then the induced counit $\kk[G_1] \to \kk$ is given by
\[ \varepsilon(\varphi) = \sum_{p \in G_0} \varphi(\sigma(p)).\]
\end{example}

\begin{example}
A special case of Example \ref{ex:groupoidvect} is the pair groupoid $([n] \times [n]) \arrows [n]$, where $[n] = \{1, \dots, n\}$, and where $\sigma$ is the diagonal map. Then the induced Frobenius algebra $\kk[[n] \times [n]]$ can be identified with the algebra of $n \times n$ matrices over $\kk$.

We note that the functor $\FSpan \to \vect_\kk$ preserves categorical products, taking disjoint unions to direct sums. Thus, by the Artin-Wedderburn theorem, it follows that, if $\kk$ is algebraically closed, then every finite-dimensional semisimple algebra arises from a Frobenius object in $\FSpan$.
\end{example}

\begin{example}
Let $X = \{a,b\}$ be a Frobenius object in $\FSpan$ of the form described in Example \ref{ex:twoelement2}. In this case, $\kk[X]$ is generated as a vector space by two elements, one of which is the unit element $1$, and the other we will call $\theta$. In the case where $\hat{\alpha}$ is the identity map and $n_{bab}=1$, $n_{bbb}=n$, the algebra multiplication is given by $\theta^2 = 1 + n\theta$, and the counit is given by $\varepsilon(1)=1$, $\varepsilon(\theta)=0$. In the case where $\hat{\alpha}$ is the nontrivial automorphism and $n_{bab}=m$, $n_{bbb}=0$, the algebra multiplication is given by $\theta^2=m$, and the counit is given by $\varepsilon(1)=0$, $\varepsilon(\theta)=1$. We remark that the case $m=0$ gives a Frobenius algebra that is isomorphic to the cohomology ring of $S^2$.
\end{example}

\bibliography{frob}
\end{document}